\documentclass[reqno]{amsart}
\usepackage{amsmath, amsthm, amssymb, amstext}

\usepackage{hyperref,xcolor}
\hypersetup{pdfborder={0 0 0},colorlinks}
\usepackage{enumitem}
\allowdisplaybreaks
\newtheorem{theorem}{Theorem}
\newtheorem{remark}[theorem]{Remark}

\newtheorem{proposition}[theorem]{Proposition}

\newtheorem{definition}[theorem]{Definition}
\newtheorem{example}[theorem]{Example}

\newcommand*\diff{\mathrm{d}}
\newcommand{\R}{\mathbb R}
\DeclareMathOperator*{\ints}{int}
\newcommand{\interior}{\ints \left(C^1(\overline{\Omega})_+\right)}

\numberwithin{theorem}{section}
\numberwithin{equation}{section}

\begin{document}

\title[Elliptic $p$-Laplacian systems with nonlinear boundary condition]{Elliptic $p$-Laplacian systems with nonlinear boundary condition}

\author[F. Borer]{Franziska Borer}
\address[F. Borer]{Technische Universit\"{a}t Berlin, Institut f\"{u}r Mathematik, Stra\ss e des 17. Juni 136, 10623 Berlin, Germany}
\email{borer@math.tu-berlin.de}

\author[S. Carl]{Siegfried Carl}
\address[S. Carl]{Institut f\"ur Mathematik, Martin-Luther-Universit\"at Halle-Wittenberg, 06099 Halle, Germany}
\email{siegfried.carl@mathematik.uni-halle.de}

\author[P. Winkert]{Patrick Winkert}
\address[P. Winkert]{Technische Universit\"{a}t Berlin, Institut f\"{u}r
Mathematik, Stra\ss e des 17. Juni 136, 10623 Berlin, Germany}
\email{winkert@math.tu-berlin.de}

\subjclass{26E25,35B38, 35J47, 35J57, 35J62,49J52}
\keywords{Clarke's gradient, elliptic systems, nonsmooth functionals, nonsmooth mountain-pass theorem, Steklov eigenvalues, sub-supersolution approach, trapping region}

\begin{abstract}
	In this paper we study quasilinear elliptic systems given by
	\begin{equation*}
		\begin{aligned}
			-\Delta_{p_1}u_1
			& =-|u_1|^{p_1-2}u_1 \quad && \text{in } \Omega,\\
			-\Delta_{p_2}u_2
			& =-|u_2|^{p_2-2}u_2 \quad && \text{in } \Omega,\\
			|\nabla u_1|^{p_1-2}\nabla u_1 \cdot \nu
			&=g_1(x,u_1,u_2) && \text{on } \partial\Omega,\\
			|\nabla u_2|^{p_2-2}\nabla u_2 \cdot \nu
			&=g_2(x,u_1,u_2)  && \text{on } \partial\Omega,
		\end{aligned}
	\end{equation*}
	where $\nu(x)$ is the outer unit normal of $\Omega$ at $x \in \partial\Omega$, $\Delta_{p_i}$ denotes the $p_i$-Laplacian and $g_i\colon \partial\Omega \times\R\times\R\to\R$ are Carath\'{e}odory functions that satisfy general growth and structure conditions for $i=1,2$. In the first part we prove the existence of a positive minimal and a negative maximal solution based on an appropriate construction of sub- and supersolution along with a certain behavior of $g_i$ near zero related to the first eigenvalue of the $p_i$-Laplacian with Steklov boundary condition. The second part is related to the existence of a third nontrivial solution by imposing a variational structure, that is, $(g_1,g_2)=\nabla g$ with a smooth function $(s_1,s_2)\mapsto g(x,s_1,s_2)$. By using the variational characterization of the second eigenvalue of the Steklov eigenvalue problem for the $p_i$-Laplacian together with the properties of the related truncated energy functionals, which are in general nonsmooth, we show the existence of a nontrivial solution whose components lie between the components of the  positive minimal and the negative maximal solution.
\end{abstract}

\maketitle

\section{Introduction}

Let $\Omega\subseteq \mathbb{R}^N$ be a bounded domain with a $C^2$-boundary $\partial \Omega$. For $i=1,2$ and $1<p_i<\infty$ we consider the following $p_i$-Laplacian system with nonlinear boundary conditions
\begin{equation}\label{problem}
	\begin{aligned}
		-\Delta_{p_1}u_1
		& =-|u_1|^{p_1-2}u_1 \quad && \text{in } \Omega,\\
		-\Delta_{p_2}u_2
		& =-|u_2|^{p_2-2}u_2 \quad && \text{in } \Omega,\\
		|\nabla u_1|^{p_1-2}\nabla u_1 \cdot \nu
		&=g_1(x,u_1,u_2) && \text{on } \partial\Omega,\\
		|\nabla u_2|^{p_2-2}\nabla u_2 \cdot \nu
		&=g_2(x,u_1,u_2)  && \text{on } \partial\Omega,
	\end{aligned}
\end{equation}
where $\nu(x)$ is the outer unit normal of $\Omega$ at $x \in \partial\Omega$, $\Delta_{p_i}$ denotes the $p_i$-Laplacian given by
\begin{align*}
	\Delta_{p_i}u_i=\operatorname{div} \left(|\nabla u_i|^{p_i-2}\nabla u_i\right)\quad\text{for }u_i \in W^{1,p_i}(\Omega), \ i=1,2,
\end{align*}
and $g_i\colon \partial\Omega \times\R\times\R\to\R$ are Carath\'{e}odory functions that satisfy appropriate growth and structure conditions, see Sections \ref{Section3} and \ref{Section4} for the detailed assumptions.

We are interested in the multiplicity of solutions of the system \eqref{problem}. In the first part, under general local conditions on the vector field $(g_1,g_2)$, we prove the existence of a positive minimal and a negative maximal solution (see Definition \ref{def-minimal-maximal-solution}) by constructing suitable pairs of sub- and supersolution to the system \eqref{problem} using a specific behavior of $g_i$ near zero corresponding to the first eigenvalue of the $p_i$-Laplacian with Steklov boundary condition (see \eqref{Steklov-eigenvalue-problem}). In the second part of this paper we suppose a variational structure of the system \eqref{problem} which means that $(g_1,g_2)=\nabla g$ with a smooth function $(s_1,s_2)\mapsto g(x,s_1,s_2)$. Then, by means of the extremal positive and negative solutions obtained in the first part, we are going to show the existence of a third nontrivial solutions whose components lie between the components of the  positive minimal and the negative maximal solution of \eqref{problem}. The proof uses a variational characterization of the second eigenvalue of the Steklov eigenvalue problem for the $p_i$-Laplacian together with the properties of the corresponding truncated energy functionals. The main difficulty is the fact that the truncated energy functionals turn out to be nonsmooth independently of the smoothness of $\nabla g$. This situation is different to the scalar case and needs further investigations in terms of Clarke's generalized gradient of locally Lipschitz functionals.

Our work is motivated by the papers of Carl-Motreanu \cite{Carl-Motreanu-2015} and Winkert \cite{Winkert-2010}. In \cite{Carl-Motreanu-2015} the authors study a Dirichlet system of the form
\begin{equation}\label{problem20}
	\begin{aligned}
		-\Delta_{p_1}u_1
		& =f_1(x,u_1,u_2) \quad && \text{in } \Omega,\\
		-\Delta_{p_2}u_2
		& =f_2(x,u_1,u_2) \quad && \text{in } \Omega,\\
		u_1
		&=u_2=0&& \text{on } \partial\Omega,
	\end{aligned}
\end{equation}
where $f_i\colon \Omega \times\R\times\R\to\R$ are Carath\'{e}odory functions having a certain local behavior near zero. It is shown that the system \eqref{problem20} has at least three nontrivial solutions whereby the first and the second eigenvalue of the $p_i$-Laplacian with Dirichlet boundary condition have been used. On the other hand, in \cite{Winkert-2010}, a scalar equation with nonlinear boundary condition of the form
\begin{equation}\label{problem21}
	\begin{aligned}
		-\Delta_{p}u
		& =f(x,u)-\lambda |u|^{p-2}u \quad && \text{in } \Omega,\\
			|\nabla u|^{p-2}\nabla u \cdot \nu
		&=\lambda |u|^{p-2}u+g(x,u) && \text{on } \partial\Omega,
	\end{aligned}
\end{equation}
has been considered. Here, the nonlinearities $f\colon\Omega\times\R\to\R$ and $g\colon\partial\Omega \times\R\to \R$ are Carath\'eodory functions which are bounded on bounded sets and which satisfy appropriate conditions near zero and at infinity. If $\lambda$ is larger than the second eigenvalue of the eigenvalue problem of the $p$-Laplacian with Steklov boundary condition, then the existence of three nontrivial solutions has been shown whereby two of them have constant sign and the third one turns out to be sign-changing. In our paper we combine the ideas of both papers to show multiplicity of solutions for the coupled system given in \eqref{problem}. We also  refer to El Manouni-Papageorgiou-Winkert \cite{ElManouni-Papageorgiou-Winkert-2015} which extends problem \eqref{problem21} to more general, nonhomogeneous operators of type $(p,q)$.

As far as we know there are only few works for elliptic systems with nonlinear boundary condition and with a variational structure. In 2016, de Godoi-Miyagaki-Rodrigues \cite{deGodoi-Miyagaki-Rodrigues-2016} studied the following Laplacian system
\begin{equation}\label{problem22}
	\begin{aligned}
		-\Delta u+C(x)u
		& =f(x,u)\quad && \text{in } \Omega,\\
		\nabla u \cdot \nu
		&=g(x,u) && \text{on } \partial\Omega,
	\end{aligned}
\end{equation}
where
\begin{align*}
	C(x)=\begin{pmatrix}
		a(x) & b(x) \\ b(x) & c(x)
	\end{pmatrix}
\end{align*}
is a positive definite matrix for a.a.\,$x\in\Omega$ and the nonlinearities $f\colon \Omega\times \R^2\to \R^2$, $g\colon\partial \Omega \times \R^2\to \R^2$ satisfy suitable growth and structure conditions. The authors prove existence results for \eqref{problem22}  when resonance or nonresonance conditions  occur by using variational tools. For systems with nonlinear boundary conditions but without a variational structure we refer to the works by Guarnotta-Livrea-Winkert \cite{Guarnotta-Livrea-Winkert-2023} who developed a sub-supersolution method for variable exponent double phase systems and Frisch-Winkert \cite{Frisch-Winkert-2023} for boundedness, existence and uniqueness results for coupled gradient dependent elliptic systems, see also the paper of Guarnotta-Marano-Moussaoui \cite{Guarnotta-Marano-Moussaoui-2022} for singular convective systems based on perturbation techniques along with fixed point arguments.

Systems with homogeneous Neumann boundary conditions have been studied in the papers by Chabrowski \cite{Chabrowski-2011} by constrained minimization based on the concentration compactness principle, by Guarnotta-Marano \cite{Guarnotta-Marano-2021-a, Guarnotta-Marano-2021-b} getting infinitely many solutions for convection problems by appropriate pairs of sub-supersolution and by Motreanu-Perera \cite{Motreanu-Perera-2009} who studied $p$-Laplace systems via Morse theory. Finally, in case of systems with Dirichlet boundary conditions, we refer to the works by
Carl-Motreanu \cite{Carl-Motreanu-2017} for convective $p$-Laplace systems based on a sub-supersolution approach, de Morais Filho-Souto \cite{deMoraisFilho-Souto-1999} using the concentration compactness principle,  Gambera-Marano-Motreanu \cite{Gambera-Marano-Motreanu-2024} for $(p,q)$-problems via Brouwer's fixed point theorem,  Hai-Shivaji \cite{Hai-Shivaji-2004} for parametric $p$-Laplacian systems, Liu-Nguyen-Winkert-Zeng \cite{Liu-Nguyen-Winkert-Zeng-2023} for
coupled double phase obstacle systems involving nonlocal functions and convection terms, Marino-Winkert \cite{Marino-Winkert-2020} for existence and uniqueness results of convection systems, Motreanu-Moussaoui-Pereira \cite{Motreanu-Moussaoui-Pereira-2018} for $p$-Laplacian systems via  sub-supersolution method and the Leray-Schauder topological degree,
Motreanu-Vetro-Vetro \cite{Motreanu-Vetro-Vetro-2016, Motreanu-Vetro-Vetro-2018} for systems involving $(p,q)$-Laplacians, see also the references therein.

The paper is organized as follows.  In Section \ref{Section2} we present the main tools which are needed in the sequel including the properties of the eigenvalue problem for the $r$-Laplacian ($1<r<\infty$) with Steklov boundary condition. Section \ref{Section3} deals with the existence of extremal positive and negative solutions where positive (resp.\,negative) means that both components are positive (resp.\,negative). Finally, in Section \ref{Section4} we are going to assume a variational structure of \eqref{problem} and prove the existence of a third nontrivial solution whose components lie between the related components of the positive and the negative solution.

\section{Preliminaries}\label{Section2}

In this section we recall the main tools that will be needed in the sequel. For $1\leq r<\infty$ we denote by $L^r(\Omega)$ and $L^r(\Omega;\R^N)$ the usual Lebesgue spaces with norm $\|\cdot\|_r$  and by $W^{1,r}(\Omega )$ the corresponding Sobolev space with norm $\|\cdot\|_{1,r}=\|\nabla \cdot\|_r+\|\cdot \|_r$. We equip the spaces $\mathcal{V}_i:=W^{1,p_i}(\Omega)$ with the equivalent norms
\begin{align*}
	\|u\|_{1,p_i}=\left(\|\nabla u\|_{p_i}^{p_i}+ \|u\|_{p_i}^{p_i}\right)^{\frac{1}{p_i}}
	\quad\text{for all }u\in\mathcal{V}_i,
\end{align*}
where $1<p_1,p_2<\infty$. Moreover, we denote by $L^{r}(\partial \Omega)$ the boundary Lebesgue space with norm $\|\cdot \|_{r,\partial\Omega}$ for any $r\in [1,\infty ]$. For $s \in \R$, we set $s^{\pm}=\max\{\pm s,0\}$ and for $u \in  W^{1, r}( \Omega )$ we define $u^{\pm}(\cdot)=u(\cdot)^{\pm}$. We have
\begin{align*}
	u^{\pm} \in  W^{1,r} ( \Omega ), \quad |u|=u^++u^-, \quad u=u^+-u^-.
\end{align*}
The space $L^{p_i}(\Omega)$ is endowed with the natural partial ordering given by the positive cone
\begin{align*}
	L^{p_i}(\Omega)_+= \left\{u\in L^{p_i}(\Omega)\colon u(x) \geq 0 \text{ a.e.\,in }\Omega\right\},
\end{align*}
which implies a related partial ordering in its subspace $W^{1,p_i}(\Omega)$. The positive cone
\begin{align*}
	\mathcal{L}_+=L^{p_1}(\Omega)_+ \times L^{p_2}(\Omega)_+
\end{align*}
induces the componentwise partial ordering on the product space
\begin{align*}
	\mathcal{L}=L^{p_1}(\Omega) \times L^{p_2}(\Omega).
\end{align*}
This implies the componentwise partial ordering in the subspace $\mathcal{W}=\mathcal{V}_1\times \mathcal{V}_2$.

\begin{definition}
	We say that $(u_1,u_2)\in\mathcal{W}$ is a weak solution of problem \eqref{problem} if
	\begin{align}\label{defsol1}
		\int_\Omega |\nabla u_1|^{p_1-2}\nabla u_1 \cdot \nabla \varphi_1\,\diff x+\int_\Omega |u_1|^{p_1-2}u_1\varphi_1\,\diff x
		=\int_{\partial\Omega} g_1(x,u_1,u_2)\varphi_1\,\diff \sigma
	\end{align}
	and
	\begin{align}\label{defsol2}
		\int_\Omega |\nabla u_2|^{p_2-2}\nabla u_2\cdot \nabla \varphi_2\,\diff x+\int_\Omega |u_2|^{p_2-2}u_2\varphi_2\,\diff x
		=\int_{\partial\Omega} g_2(x,u_1,u_2)\varphi_2\,\diff \sigma
	\end{align}
	hold true for all $(\varphi_1,\varphi_2)\in\mathcal{W}$ and all the integrals in \eqref{defsol1} and \eqref{defsol2} are finite. Here, $\sigma$ denotes the $(N-1)$-dimensional Hausdorff surface measure on $\partial\Omega$.
\end{definition}

Next, we introduce the notion of weak sub- and supersolution to \eqref{problem}.

\begin{definition}\label{def-sub-supersolution}
	We say that $(\underline{u}_1,\underline{u}_2)$, $(\overline{u}_1, \overline{u}_2)\in \mathcal{W}$ form a pair of weak sub- and supersolution of problem \eqref{problem} if $\underline{u}_i\leq \overline{u}_i$ a.e.\,in $\Omega$ for $i=1,2$ and
	\begin{equation}\label{defsubsuper1}
		\begin{aligned}
			&\int_\Omega \left(|\nabla \underline{u}_1|^{p_1-2}\nabla \underline{u}_1\cdot \nabla \varphi_1+|\underline{u}_1|^{p_1-2}\underline{u}_1\varphi_1\right)\diff x
			-\int_{\partial\Omega} g_1(x,\underline{u}_1,w_2)\varphi_1\,\diff \sigma\\
			&+\int_\Omega \left(|\nabla \underline{u}_2|^{p_2-2}\nabla \underline{u}_2\cdot \nabla \varphi_2+|\underline{u}_2|^{p_2-2}\underline{u}_2\varphi_2
			\right)\diff x\\
			&-\int_{\partial\Omega} g_2(x,w_1,\underline{u}_2)\varphi_2\,\diff \sigma \leq 0
		\end{aligned}
	\end{equation}
	and
	\begin{equation}\label{defsubsuper2}
		\begin{aligned}
			&\int_\Omega \left(|\nabla \overline{u}_1|^{p_1-2}\nabla \overline{u}_1\cdot \nabla \varphi_1+ |\overline{u}_1|^{p_1-2}\overline{u}_1\varphi_1
			\right) \diff x
			-\int_{\partial\Omega} g_1(x,\overline{u}_1,w_2)\varphi_1\,\diff \sigma\\
			&+\int_\Omega \left(|\nabla \overline{u}_2|^{p_2-2}\nabla \overline{u}_2\cdot \nabla \varphi_2+ |\overline{u}_2|^{p_2-2}\overline{u}_2\varphi_2
			\right)\diff x\\
			&-\int_{\partial\Omega} g_2(x,w_1,\overline{u}_2)\varphi_2\,\diff \sigma \geq 0
		\end{aligned}
	\end{equation}
	for all $(\varphi_1,\varphi_2)\in\mathcal{W}$ with $\varphi_1,\varphi_2\geq 0$ a.e.\,in $\Omega$ and for all $(w_1,w_2)\in \mathcal{W}$ such that $\underline{u}_i\leq w_i\leq \overline{u}_i$  for $i=1,2$ and with all integrals in \eqref{defsubsuper1} and \eqref{defsubsuper2} to be finite.

	If $\underline{u}=(\underline{u}_1,\underline{u}_2)$, $\overline{u}=(\overline{u}_1,\overline{u}_2)$ is a pair of weak sub- and supersolution, then the order interval $[\underline{u},\overline{u}]=[\underline{u}_1,\overline{u}_1] \times [\underline{u}_2,\overline{u}_2]$ is called trapping region, whereby
	\begin{align*}
		[\underline{u}_i,\overline{u}_i]
		=\big\{u\in W^{1,p_i}(\Omega)\,:\,\underline{u}_i\leq u\leq \overline{u}_i\text{ a.e.\,in }\Omega\big\}.
	\end{align*}
\end{definition}

For $1<p_i<\infty$, $i=1,2$, let $A_{p_i}\colon \mathcal{V}_i\to \mathcal{V}_i^*$ be the operator given by
\begin{align}\label{operator-Api}
	\langle A_{p_i}(u_i),\varphi_i\rangle_{\mathcal{V}_i}=\int_{\Omega} |\nabla u_i|^{p_i-2}\nabla u_i\cdot \nabla \varphi_i \,\diff x
\end{align}
for $u_i,\varphi_i\in \mathcal{V}_i$, where $\langle \,\cdot\,,\,\cdot\,\rangle_{\mathcal{V}_i}$ denotes the duality pairing between $\mathcal{V}_i$ and its dual space $\mathcal{V}_i^*$. The following proposition summarizes the main properties of $A_{p_i}$, see, for example, Carl-Le-Motreanu \cite[Lemma 2.111]{Carl-Le-Motreanu-2007}.

\begin{proposition}\label{proposition-properties-operator}
	Let $p_i\in (1,\infty)$ and let $A_{p_i}\colon \mathcal{V}_i\to \mathcal{V}_i^*$  be given by \eqref{operator-Api}. Then $A_{p_i}$ is well-defined, bounded, continuous, monotone and of type \textnormal{(S$_+$)}, that is, $u_i^k\rightharpoonup u_i$ in $\mathcal{V}_i$ and $\limsup_{k\rightarrow\infty}\,\langle A_{p_i} (u_i^k), u_i^k-u_i\rangle\leq 0$ imply $u_i^k\rightarrow u_i$ in $\mathcal{V}_i$ for $i=1,2$.
\end{proposition}

Next, we want to explain the notion of minimal and maximal constant sign solutions.

\begin{definition}\label{def-minimal-maximal-solution}
	An element $m \in \mathcal{W}$ is said to be a minimal positive solution of \eqref{problem} if $m$ is a positive solution of \eqref{problem} and if for any positive solution $u$ with $u \leq m$ it follows that $m=u$. Similarly, we define a maximal negative solution.
\end{definition}

Let $C^1(\overline{\Omega})$ be equipped with norm $\|\cdot\|_{C^1(\overline{\Omega})}$ and let $C^1(\overline{\Omega})_+$ be its positive cone defined by
\begin{align*}
	C^1(\overline{\Omega})_+=\left\{u \in C^1(\overline{\Omega}): u(x) \geq 0 \text{ for all } x \in \overline{\Omega}\right\}.
\end{align*}
This cone has a nonempty interior given by
\begin{align*}
	\interior=\left\{u \in C^1(\overline{\Omega})_+: u(x)>0 \text{ for all } x \in \overline{\Omega}\right\}.
\end{align*}

Let us recall some basic facts about the Steklov eigenvalue problem for the $r$-Laplacian with $r\in (1,\infty)$ which is given by
\begin{equation}\label{Steklov-eigenvalue-problem}
	\begin{aligned}
		-\Delta_r u
		& =-|u|^{r-2}u \quad &  & \text{in }\Omega, \\
		|\nabla u|^{r-2}\nabla u\cdot \nu
		& =\lambda |u|^{r-2}u& & \text{on }\partial\Omega.
	\end{aligned}
\end{equation}
From L\^{e} \cite{Le-2006} we know that the set of eigenvalues of \eqref{Steklov-eigenvalue-problem}, denoted by $\sigma(r)$, has a smallest element $\lambda_{1,r}$ which is positive, isolated, simple and can be characterized by
\begin{align*}
	\lambda_{1,r}=\inf_{u\in W^{1,r}(\Omega)}{\left\{\|\nabla u\|_r^r+\|u\|_{r}^r \,:\,\|u\|_{r,\partial\Omega}^r=1\right\}}.
\end{align*}
We further point out that every eigenfunction corresponding to the first eigenvalue $\lambda_{1,r}$ does not change sign in $\overline{\Omega}$. In fact it turns out that every eigenfunction associated to an eigenvalue $\lambda \neq \lambda_{1,r}$ changes sign on $\partial \Omega$.

In what follows we denote by $u_{1,r}$ the normalized (i.e., $\|u_{1,r}\|_{r, \partial \Omega}=1$) positive eigenfunction corresponding to $\lambda_{1,r}$. As shown in L\^{e} \cite{Le-2006}, thanks to the nonlinear regularity theory and the nonlinear maximum principle, we can suppose that $u_{1,r} \in \interior$. Additionally, due to the fact that $\lambda_{1,r}$ is isolated, the second eigenvalue $\lambda_{2,r}$ is well-defined by
\begin{align*}
	\lambda_{2,r}=\inf \left[\lambda \in \sigma(r)\,:\, \lambda>\lambda_{1,r} \right].
\end{align*}

Now, let $\partial B^{r, \partial \Omega}_1=\{u \in L^r(\partial\Omega)\,:\, \|u\|_{r, \partial \Omega}=1\}$ and $S_r=W^{1,r}(\Omega) \cap \partial B_1^{r, \partial \Omega}$. Then, due to Mart\'{\i}nez-Rossi \cite{Martinez-Rossi-2004}, we have a variational characterization of $\lambda_{2,r}$ given by
\begin{align*}
	\lambda_{2,r}= \inf_{\hat{\gamma} \in \hat{\Gamma}(r)} \max_{-1 \leq t \leq 1} \Big[ \|\nabla \hat{\gamma}(t)\|_r^r+\|\hat{\gamma}(t)\|_r^r \Big],
\end{align*}
where $\hat{\Gamma}(r)=\{ \hat{\gamma} \in C([0,1],S_r)\,:\, \hat{\gamma}(0)=-u_{1,r},\, \hat{\gamma}(1)=u_{1,r} \}$.

Next, we recall some basic notions in nonsmooth analysis that are required in the sequel. We refer to the monograph of Carl-Le-Motreanu \cite{Carl-Le-Motreanu-2007}. For a real Banach space $(X,\|\cdot\|_X)$, we denote by $X^*$ its dual space and by $\langle \cdot , \cdot \rangle$ the duality pairing between $X$ and $X^*$. A function $f\colon X\to\mathbb{R}$ is said to be locally Lipschitz if for every $x\in X$ there exist a neighborhood $U_x$ of $x$ and a
constant $L_x\geq 0$ such that
\begin{align*}
	| f(y)-f(z)| \leq L_x \| y-z\|_X \quad \text{for all } y,z \in U_x.
\end{align*}
For a locally Lipschitz function $f\colon X\to \mathbb{R}$ on a Banach space $X$, the generalized directional derivative of $f$ at the point $x\in X$ along the direction $y\in X$ is defined by
\begin{align*}
	f^{\circ}(x;y):=\limsup_{z\to x, t\to 0^+}
	\frac{f(z+ty)-f(z)}{t},
\end{align*}
see Clarke \cite[Chapter 2]{Clarke-1990}. Note that if $f\colon  X\to \mathbb{R}$ is strictly differentiable, that is, for all $x\in X$, $f'(x) \in X^*$ exists such that
\begin{align*}
	\lim_{\substack{z\to x\\ t \to 0^+}} \frac{f(z+ty)-f(z)}{t}=\langle f'(x),y\rangle \quad \text{for all } y \in X,
\end{align*}
then the usual directional derivative $f'(x;y)$ given by
\begin{align*}
	f'(x;y)=\lim_{t\to 0^+}\frac{f(x+ty)-f(x)}{t}
\end{align*}
exists and coincides with the generalized directional derivative $f^\circ(x;y)$.

If $f_1$, $f_2\colon  X \to \mathbb{R}$ are locally Lipschitz functions, then we have
\begin{align*}
	(f_1+f_2)^{\circ}(x;y)\leq f_1^{\circ}(x;y)+f_2^{\circ}(x;y)\quad \text{for all } x,y \in X.
\end{align*}
The generalized gradient of a locally Lipschitz function $f:X\to\mathbb{R}$ at $x\in X$ is the set
\begin{equation*}
	\partial f(x):=\left\{ x^*\in X^*\colon \langle x^*,y\rangle\leq f^{\circ}(x;y) \quad \text{for all } y \in X \right\}.
\end{equation*}
Based on the Hahn-Banach theorem we easily verify that $\partial f(x)$ is nonempty. An element $x\in X$ is said to be a critical point of  a locally Lipschitz function $f\colon X\to \mathbb{R}$ if there holds
\begin{equation*}
	f^{\circ}(x;y)\geq 0 \quad \text{for all $y \in X$}
\end{equation*}
or, equivalently, $0\in\partial f(x)$, see Chang \cite{Chang-1981}.

The nonsmooth mountain-pass theorem due to Chang is stated as follows \cite[Theorem 3.4]{Chang-1981}.

\begin{theorem}\label{mountain-pass-theorem}
	Let $X$ be a reflexive real Banach space and let $J \colon X \to \R$ be a locally Lipschitz functional satisfying the nonsmooth Palais-Smale condition. If there exist $x_0, x_1 \in X$ and a constant $r > 0$ such that $\|x_1-x_0\|>r$ and $\max\{J(x_0),J(x_1)\} <\inf_{x \in \partial B_r(x_0)} J(x)$, then $J$ has a critical point $u_0 \in X$ such that
	\begin{align*}
		\inf_{x \in \partial B_r(x_0) } J(x) \leq J(u_0)=\inf_{\pi \in \Pi} \max_{t \in [0,1]} J(\pi(t)),
	\end{align*}
	where $\Pi=\{ \pi \in C([0,1],X )\colon \pi(0)=x_0 , \pi(1)=x_1 \}$ and $\partial B_r(x_0)=\{u \in X\colon \|u-x_0\|=r\}$.
\end{theorem}

\section{Constant-sign solutions}\label{Section3}

In this section we prove the existence of maximal and minimal constant sign solutions for problem \eqref{problem}. We suppose the following hypotheses:
\begin{enumerate}[label=\textnormal{(H$_0$)},ref=\textnormal{H$_0$}]
	\item\label{H0}
	For $i=1,2$, the functions $g_i\colon \partial\Omega \times\R\times\R\to\R$ are Carath\'{e}odory functions such that $g_i(x,0,0)=0$ for a.a.\,$x\in\partial\Omega$ and
	\begin{align*}
		&|g_i(x,s_1,s_2)  |\le H_i(x)\quad\text{for a.a.\,}x\in\partial\Omega,
	\end{align*}
	for all $(s_1,s_2)\in M$, whereby $M$ is a bounded set and $H_i\in L^\infty(\partial\Omega)$. Moreover, it holds
	\begin{equation}\label{condition-loc-Hoelder}
		\begin{aligned}
			&|g_i(x_1,s_1,t_1)-g_i(x_2,s_2,t_2)|\\
			&\leq L_i \left(|x_1-x_2|^{\alpha_i}+|s_1-s_2|^{\alpha_i}+|t_1-t_2|^{\alpha_i}\right)
		\end{aligned}
	\end{equation}
	for all $(x_1,s_1,t_2), (x_2,s_2,t_2) \in \partial\Omega \times [-K_i,K_i]\times [-K_i,K_i]$, where $K_i$ is a positive constant, $\alpha_i \in (0,1]$ and $\|H_i\|_{\infty,\partial\Omega} \leq L_i$.
\end{enumerate}
\begin{enumerate}[label=\textnormal{(H$_1$)},ref=\textnormal{H$_1$}]
	\item\label{H1}
	There exist constants $k_i>0$ and $d_i<0$ for $i=1,2$ such that
	\begin{align*}
		g_1(x,k_1,s_2) &\leq 0\quad \text{for a.a.\,}x\in\partial\Omega \text{ and for all }s_2\in [0,k_2],\\
		g_1(x,d_1,s_2) &\geq 0\quad \text{for a.a.\,}x\in\partial\Omega \text{ and for all }s_2\in [d_2,0],\\
		g_2(x,s_1,k_2) &\leq 0\quad \text{for a.a.\,}x\in\partial\Omega \text{ and for all }s_1\in [0,k_1],\\
		g_2(x,s_1,d_2) &\geq 0\quad \text{for a.a.\,}x\in\partial\Omega \text{ and for all }s_1\in [d_1,0].
	\end{align*}
\end{enumerate}

\begin{enumerate}[label=\textnormal{(H$_2$)},ref=\textnormal{H$_2$}]
	\item\label{H2}
	For $i=1,2$, there exist constants $c_i>\lambda_{1,p_i}$ such that
	\begin{align*}
		\liminf_{s_1\to 0^+} \frac{g_1(x,s_1,s_2)}{s_1^{p_1-1}} \geq c_1
	\end{align*}
	uniformly for a.a.\,$x\in\partial\Omega$ and for all $s_2\in (0,k_2]$,
	\begin{align*}
		\liminf_{s_1\to 0^-} \frac{g_1(x,s_1,s_2)}{|s_1|^{p_1-2}s_1} \geq c_1
	\end{align*}
	uniformly for a.a.\,$x\in\partial\Omega$ and for all $s_2\in [d_2,0)$,
	\begin{align*}
		\liminf_{s_2\to 0^+} \frac{g_2(x,s_1,s_2)}{s_2^{p_2-1}} \geq c_2
	\end{align*}
	uniformly for a.a.\,$x\in\partial\Omega$ and for all $s_1\in (0,k_1]$,
	\begin{align*}
		\liminf_{s_2\to 0^-} \frac{g_2(x,s_1,s_2)}{|s_2|^{p_2-2}s_2} \geq c_2
	\end{align*}
	uniformly for a.a.\,$x\in\partial\Omega$ and for all $s_1\in [d_1,0)$.
\end{enumerate}

\begin{remark}\label{regularity-argument}
	Note that \eqref{condition-loc-Hoelder} is needed for the usage of the regularity results of Lieberman \cite{Lieberman-1988}.
	Indeed, if $u=(u_1,u_2)$ is a solution of \eqref{problem} such that $(0,0)\le (u_1,u_2) \le (k_1,k_2)$ and both not identically zero, then $(u_1,u_2) \in \interior \times \interior$. Let us verify this just for $u_1$, the case for $u_2$ works in the same way. First, from the boundedness of $u_1$ and \eqref{condition-loc-Hoelder} along with Theorem 2 in Lieberman \cite{Lieberman-1988}, we know that $u\in C^{1,\alpha}(\overline{\Omega})$ for some $\alpha\in (0,1)$. From the first line in \eqref{problem}  we have $\Delta_{p_1} u_1 \leq  u_1^{p_1-1}$ for a.a.\,$x\in\Omega$. Taking $\beta(s)=s^{p_1-1}$ for all $s>0$, we get from V\'{a}zquez's strong maximum principle (see \cite{Vazquez-1984}) that $u_1(x)>0$ in $\Omega$ since $\int_{0^+} \frac{1}{(s\beta(s))^{\frac{1}{p_1}}}\,\mathrm{d}s=+\infty$.
	Suppose there exists $x_0 \in \partial\Omega$ such that $u_1(x_0)=0$. Applying again the maximum principle we obtain $\nabla u_1(x_0)\cdot \nu(x_0)<0$. In view of hypothesis \eqref{H2} first line, for $\varepsilon>0$ small enough such that $c_1-\varepsilon>0$, there exists $\delta>0$ such that for all $s_1\in (0,\delta)$ we get
	\begin{align*}
		g_1(x_0, s_1, s_2)\ge (c_1-\varepsilon) s_1^{p_1-1}\quad \text{uniformly for all } s_2\in (0,k_2],
	\end{align*}
	which yields by the continuity of $g_1$ as $s_1\to 0^+$
	\begin{align*}
		g_1(x_0, 0, s_2)\ge 0\quad \text{uniformly for all } s_2\in (0,k_2].
	\end{align*}
	The continuity of $g_1$ then shows that $g_1(x_0, 0, s_2)\ge 0$ for all  $s_2\in [0,k_2]$, in particular for $s_2= u_2(x_0)\in [0, k_2]$, that is, we have $g_1(x_0, 0, u_2(x_0))\ge 0$, and thus from the third line of \eqref{problem} it follows
	\begin{align*}
		\nabla u_1(x_0)\cdot \nu(x_0)\ge 0,
	\end{align*}
	which is in contradiction to $\nabla u_1(x_0)\cdot \nu(x_0)<0$.
	Hence, $u_1>0$ in $\overline{\Omega}$ and so $u_1\in \interior$. A similar result holds for a solution $(v_1,v_2)$ such that $(d_1,d_2)\le (v_1,v_2)\le (0,0)$, both not identically zero, then $(v_1,v_2) \in (-\interior) \times (-\interior)$.
\end{remark}

\begin{theorem}\label{theorem-constant-sign-solutions}
	Let hypotheses \eqref{H0}, \eqref{H1} and \eqref{H2} be satisfied. Then there exist a positive solution $(u_1,u_2)\in\mathcal{W}$ and a negative solution $(v_1,v_2)\in\mathcal{W}$ of the system \eqref{problem}.
\end{theorem}

\begin{proof}
	From \eqref{H1} we directly obtain
	\begin{equation}\label{proof-1}
		\begin{aligned}
			-g_1(x,k_1,s_2) &\geq 0\quad \text{for a.a.\,}x\in\partial\Omega \text{ and for all }s_2\in [0,k_2],\\
			-g_2(x,s_1,k_2) &\geq 0\quad \text{for a.a.\,}x\in\partial\Omega \text{ and for all }s_1\in [0,k_1].
		\end{aligned}
	\end{equation}

	Hypothesis \eqref{H2} implies that there exists $\delta \in (0,\min\{k_1,k_2\})$ such that
	\begin{align}\phantom{\eqref{estimate-3}}\label{estimate-3}
		g_1(x,s_1,s_2)> \lambda_{1,p_1} s_1^{p_1-1}
	\end{align}
	for a.a.\,$x\in\partial \Omega$, for all $s_1\in(0,\delta)$ and for all $s_2\in (0,k_2]$,
	\begin{align}\label{estimate-4}
		g_2(x,s_1,s_2)> \lambda_{1,p_2} s_2^{p_2-1}
	\end{align}
	for a.a.\,$x\in\partial\Omega$, for all $s_1\in(0,k_1]$ and for all $s_2\in(0,\delta)$.

	From the Steklov eigenvalue problem for the $p_i$-Laplacian multiplied with $\varepsilon^{p_i-1}>0$ we know that
	\begin{equation}\label{Robin2}
		\begin{aligned}
			&\int_\Omega |\nabla (\varepsilon u_{1,p_i})|^{p_i-2}\nabla (\varepsilon u_{1,p_i})\cdot \nabla \varphi_i\,\diff x+\int_{\Omega} (\varepsilon u_{1,p_i})^{p_i-1}\varphi_i\,\diff x\\
			&=\lambda_{1,p_i} \int_{\partial \Omega} (\varepsilon u_{1,p_i})^{p_i-1}\varphi_i\,\diff \sigma
		\end{aligned}
	\end{equation}
	holds for all $\varphi_i \in \mathcal{V}_i$ with $\varphi_i\geq 0$ and $i=1,2$. We choose $\varepsilon>0$ small enough such that
	\begin{align}\label{eigenfunction-1}
		\varepsilon u_{1,p_i}(x)< \delta\quad\text{for all }x\in\Omega \text{ and } i=1,2.
	\end{align}
	Using \eqref{Robin2} and \eqref{eigenfunction-1} along with \eqref{estimate-3} and \eqref{estimate-4} in \eqref{defsubsuper1} for
	\begin{align*}
		(\underline{u}_1,\underline{u}_2):=(\varepsilon u_{1,p_1},\varepsilon u_{1,p_2})
		\quad\text{and}\quad
		(\overline{u}_1,\overline{u}_2):=(k_1,k_2)
	\end{align*}
	gives
	\begin{align*}
		&\int_{\partial\Omega} \left(\lambda_{1,p_1}(\varepsilon u_{1,p_1})^{p_1-1}-g_1(x,\varepsilon u_{1,p_1},w_2)\right) \varphi_1\,\diff \sigma\\
		&+\int_{\partial\Omega} \left(\lambda_{1,p_2}(\varepsilon u_{1,p_2})^{p_2-1}-g_2(x,w_1,\varepsilon u_{1,p_2})\right)\varphi_2\,\diff \sigma\leq 0.
	\end{align*}
	On the other hand, we get from \eqref{proof-1} and \eqref{defsubsuper2} that
	\begin{align*}
			&\int_\Omega \left(|\nabla k_1|^{p_1-2}\nabla k_1\cdot \nabla \varphi_1+k_1^{p_1-1}\varphi_1
			\right) \diff x\
			+\int_{\partial\Omega} \left(-g_1(x,k_1,w_2)\right)\varphi_1\,\diff \sigma\\
			&+\int_\Omega \left(|\nabla k_2|^{p_2-2}\nabla k_2\cdot \nabla \varphi_2+ k_2^{p_2-1}\varphi_2
			\right)\diff x
			+\int_{\partial\Omega} \left(-g_2(x,w_1,k_2)\right) \varphi_2\,\diff \sigma \geq 0
	\end{align*}
	for all $(\varphi_1,\varphi_2)\in\mathcal{W}$ with $\varphi_1,\varphi_2\geq 0$ a.e.\,in $\Omega$ and for all $(w_1,w_2)\in \mathcal{W}$ such that $\underline{u}_i\leq w_i\leq \overline{u}_i$  for $i=1,2$. Therefore, $(\underline{u}_1,\underline{u}_2)\in\mathcal{W}$ and $(\overline{u}_1,\overline{u}_2)\in\mathcal{W}$ form a pair of sub- and supersolution related to Definition \ref{def-sub-supersolution}. From Guarnotta-Livrea-Winkert \cite{Guarnotta-Livrea-Winkert-2023} (for $\mu\equiv 0$) we know that a solution $(u_1,u_2)\in\mathcal{W}$ of the system \eqref{problem} exists such that $u_i \leq k_i$. Moreover, the nonlinear regularity theory implies that $(u_1,u_2) \in \operatorname{int}(C^1(\overline{\Omega})_+)\times \operatorname{int}(C^1(\overline{\Omega})_+)$, see Remark \ref{regularity-argument}.

	Similarly, one can show that $(d_1,d_2)$ and $(-\varepsilon u_{1,p_1},-\varepsilon u_{1,p_2})$ form a pair of sub- and supersolution in the sense of Definition \ref{def-sub-supersolution} for the system \eqref{problem} provided the parameter $\varepsilon>0$ is sufficiently small. Therefore, we obtain a negative solution $(v_1,v_2) \in (-\operatorname{int}(C^1(\overline{\Omega})_+))\times (-\operatorname{int}(C^1(\overline{\Omega})_+))$ satisfying $v_i \geq d_i$ for $i=1,2$.
\end{proof}

Next, we are going to prove the existence of a minimal positive and of a maximal negative solution of the system \eqref{problem} in the trapping region constructed in the proof of Theorem \ref{theorem-constant-sign-solutions}.

\begin{theorem}\label{theorem-minimal-maximal-constant-sign-solutions-away-from-zero}
	Let hypotheses \eqref{H0}, \eqref{H1} and \eqref{H2} be satisfied. Then, for a given solution $(u_1,u_2) \in\mathcal{W}$ of problem \eqref{problem} in $[\varepsilon u_{1,p_1},k_1]\times [\varepsilon u_{1,p_2},k_2]$ for some $\varepsilon>0$ there exists a minimal solution $(u_1^\varepsilon,u_2^\varepsilon)$ of \eqref{problem} in $[\varepsilon u_{1,p_1},k_1]\times [\varepsilon u_{1,p_2},k_2]$ such that $u_i^\varepsilon \leq u_i$ for $i=1,2$. Furthermore, given a solution $(v_1,v_2)\in \mathcal{W}$ of problem \eqref{problem} in $[d_1,-\varepsilon u_{1,p_1}]\times [d_2,-\varepsilon u_{1,p_2}]$ for some $\varepsilon>0$, there exists a maximal solution $(v_1^\varepsilon,v_2^\varepsilon)$ of \eqref{problem} in $[d_1,-\varepsilon u_{1,p_1}]\times [d_2,-\varepsilon u_{1,p_2}]$ such that $v_i^\varepsilon \geq v_i$ for $i=1,2$.
\end{theorem}

\begin{proof}
We are going to prove just the first assertion of the theorem, the second one can be shown using similar arguments.

We choose $\varepsilon>0$ sufficiently small (like in the proof of Theorem \ref{theorem-constant-sign-solutions}).
Then, Theorem \ref{theorem-constant-sign-solutions} guarantees that a solution $(u_1,u_2)\in \mathcal{W}$ of \eqref{problem} exists in $[\varepsilon u_{1,p_1},k_1]\times [\varepsilon u_{1,p_2},k_2]$. Denote by $\mathcal{S}_\varepsilon$ the set of all solutions $(h_1,h_2)$ of \eqref{problem} such that  $(h_1,h_2) \in [\varepsilon u_{1,p_1},k_1]\times [\varepsilon u_{1,p_2},k_2]$ satisfying $h_i\leq u_i$ for $i=1,2$. Apparently, $S_\varepsilon$ is not empty. We are going to prove that $\mathcal{S}_\varepsilon$ has a minimal element by applying Zorn's Lemma. For this purpose, let $\mathcal{C}$ be a chain in $\mathcal{S}_\varepsilon$. Then we can find a sequence $\{u_1^k,u_2^k\}_{k\geq 1} \subset \mathcal{C}$ such that $u_i^{k+1} \leq u_i^k$ for $i=1,2$ and for all $k\geq 1$ satisfying
	\begin{align*}
		\inf \mathcal{C}=\inf_{k\geq 1} (u_1^k,u_2^k).
	\end{align*}
	Since $(u_1^k,u_2^k) \in \mathcal{C}$ we know that $(u_1^k,u_2^k)$ solves system \eqref{problem}. Testing \eqref{defsol1} with $\varphi_1=u_1^k$ and \eqref{defsol2} with $\varphi_2=u_2^k$ and using \eqref{H0} together with the trace theorem, we get that
	\begin{align*}
		\|u_i^k\|_{1,p_i}^{p_i-1} \leq C_i
	\end{align*}
	for $C_i>0$ independent of $u_i^k$ and for all $u_i^k \in \mathcal{V}_i$. Hence, the sequence $\{u_1^k,u_2^k\}_{k\geq 1}$ is bounded in $\mathcal{W}$. Therefore, up to a subsequence if necessary, not relabeled, we may assume that
		\begin{equation}\label{convergence-properties}
		\begin{aligned}
			u_i^k &\rightharpoonup \hat{u}_i\quad &&\text{in }\mathcal{V}_i,\quad i=1,2,\\
			u_i^k(x)&\to \hat{u}_i(x)\quad &&\text{for a.a.\,}x\in\Omega\\
			u_i^k(x)&\to \hat{u}_i(x)\quad &&\text{for a.a.\,}x\in\partial\Omega.
		\end{aligned}
	\end{equation}
	From \eqref{convergence-properties} we conclude that $(\hat{u}_1,\hat{u}_2)\in [\varepsilon u_{1,p_1},k_1]\times [\varepsilon u_{1,p_2},k_2]$ and $\hat{u}_i\leq u_i$ for $i=1,2$. Furthermore, testing the corresponding weak formulations with $u_i^k-\hat{u}_i$ and using \eqref{convergence-properties} along with $\eqref{H0}$ we get that
	\begin{align*}
		\limsup_{k\to\infty}\,\langle A_{p_i} (u_i^k), u_i^k-\hat{u}_i\rangle\leq 0\quad\text{for }i=1,2.
	\end{align*}
	Combining this with \eqref{convergence-properties} and the fact that $A_{p_i}$ fulfills the \textnormal{(S$_+$)}-property on $\mathcal{V}_i$, see Proposition \ref{proposition-properties-operator}, we conclude that
	\begin{align}\label{convergence-properties-2}
		u_i^k \to \hat{u}_i \quad \text{in }\mathcal{V}_i,\quad i=1,2.
	\end{align}
	Applying \eqref{convergence-properties-2} to the corresponding weak formulations shows that $(\hat{u}_1,\hat{u}_2)$ is a solution of \eqref{problem} that belongs to $\mathcal{S}_\varepsilon$ and $\inf \mathcal{C}=(\hat{u}_1,\hat{u}_2)\in \mathcal{S}_\varepsilon$. From Zorn's Lemma, see Papageorgiou-Winkert \cite[p.\,36]{Papageorgiou-Winkert-2018}, we conclude that $\mathcal{S}_\varepsilon$ has a minimal element $(u_1^\varepsilon,u_2^\varepsilon)$.
\end{proof}

In order to get maximal and minimal solutions of \eqref{problem}, we have to suppose further conditions on the vector field $(g_1,g_2)$ near zero as follows.
\begin{enumerate}[label=\textnormal{(H$_3$)},ref=\textnormal{H$_3$}]
	\item\label{H3}
	There exist constants $\alpha_i\geq c_i$, $i=1,2$, such that
	\begin{align*}
		\limsup_{s_1\to 0^+} \frac{g_1(x,s_1,s_2)}{s_1^{p_1-1}} \leq \alpha_1
	\end{align*}
	uniformly for a.a.\,$x\in\partial\Omega$ and for all $s_2\in (0,k_2]$,
	\begin{align*}
		\limsup_{s_1\to 0^-} \frac{g_1(x,s_1,s_2)}{|s_1|^{p_1-2}s_1} \leq \alpha_1
	\end{align*}
	uniformly for a.a.\,$x\in\partial\Omega$ and for all $s_2\in [d_2,0)$,
	\begin{align*}
		\limsup_{s_2\to 0^+} \frac{g_2(x,s_1,s_2)}{s_2^{p_2-1}} \leq \alpha_2
	\end{align*}
	uniformly for a.a.\,$x\in\partial\Omega$ and for all $s_1\in (0,k_1]$,
	\begin{align*}
		\limsup_{s_2\to 0^-} \frac{g_2(x,s_1,s_2)}{|s_2|^{p_2-2}s_2} \leq \alpha_2
	\end{align*}
	uniformly for a.a.\,$x\in\partial\Omega$ and for all $s_1\in [d_1,0)$.
\end{enumerate}

Now we can state and prove our main result on maximal and minimal solutions of \eqref{problem}.

\begin{theorem}\label{theorem-minimal-maximal-constant-sign-solutions}
	Let hypotheses \eqref{H0}--\eqref{H3} be satisfied. Then, problem \eqref{problem} admits a positive solution $(u_{1,+},u_{2,+})\in \operatorname{int}(C^1(\overline{\Omega})_+)\times \operatorname{int}(C^1(\overline{\Omega})_+)$ such that $u_{i,+}\leq k_i$ for $i=1,2$, which is minimal among the positive solutions of \eqref{problem}. Moreover, problem \eqref{problem} admits a negative solution $(u_{1,-},u_{2,-})\in (-\operatorname{int}(C^1(\overline{\Omega})_+))\times (-\operatorname{int}(C^1(\overline{\Omega})_+))$ such that $u_{i,-}\geq d_i$ for $i=1,2$, which is maximal among the negative solutions of \eqref{problem}.
\end{theorem}

\begin{proof}
	As before, we only show the existence of a minimal positive solution of \eqref{problem}, the proof for the maximal negative solution works in a similar way. The application of Theorems \ref{theorem-constant-sign-solutions} and \ref{theorem-minimal-maximal-constant-sign-solutions-away-from-zero} gives us a sequence $\{(u_1^n,u_2^n)\}_{n \geq n_0}\subseteq \mathcal{W}$ for $n_0$ sufficiently large such that for every integer $n\geq n_0$ we have that $(u_1^n,u_2^n)$ is a solution of \eqref{problem} that is minimal in the trapping region $[\frac{1}{n} u_{1,p_1},k_1]\times [\frac{1}{n} u_{1,p_2},k_2]$ such that $u_i^{n+1} \leq u_i^n$ for $i=1,2$. From this and \eqref{H0} we may suppose, for a subsequence if necessary, not relabeled, that, for $i=1,2$,
	\begin{equation*}
		\begin{aligned}
			u_i^n &\rightharpoonup u_{i,+}\quad &&\text{in }\mathcal{V}_i,\\
			u_i^n&\to u_{i,+}\quad &&\text{in } L^{p_i}(\Omega) \text{ and  pointwisely a.e.\,in }\Omega,\\
			u_i^n&\to u_{i,+}\quad &&\text{in } L^{p_i}(\partial\Omega) \text{ and pointwisely a.e.\,in }\partial\Omega.
		\end{aligned}
	\end{equation*}
	for some $(u_{1,+},u_{2,+})\in \mathcal{W}$. As in the proof of Theorem \ref{theorem-minimal-maximal-constant-sign-solutions-away-from-zero} by applying the \textnormal{(S$_+$)}-property of $A_{p_i}$  on $\mathcal{V}_i$, see Proposition \ref{proposition-properties-operator}, we conclude that $(u_{1,+},u_{2,+})$ is a solution of \eqref{problem}.\\
	{\bf Claim :} $u_{i,+}\neq 0$ for $i=1,2$.\\
	Suppose this is not the case and assume that $u_{1,+}=0$. For each $n \geq n_0$ we set
	\begin{align*}
		w_n&=\frac{u_1^n}{\|u_1^n\|_{1,p_1}}
		\quad\text{and}\quad
		\xi_n=\frac{g_1(x,u_1^n,u_2^n)}{(u_1^n)^{p_1-1}}w_n^{p_1-1}.
	\end{align*}
	Clearly the sequence $\{w_n\}_{n \geq n_0} \subseteq \mathcal{V}_1$ is bounded and due to hypotheses \eqref{H2} and \eqref{H3} we may assume that
	\begin{equation}\label{convergence-properties3}
		\begin{aligned}
			w_n &\rightharpoonup w \quad &&\text{in }\mathcal{V}_1,\\
			w_n(x)&\to w(x)\quad &&\text{in } L^{p_1}(\Omega) \text{ and pointwisely a.e.\,in }\Omega,\\
			w_n(x)&\to w(x)\quad &&\text{in } L^{p_1}(\partial \Omega) \text{ and pointwisely a.e.\,in } \partial\Omega,\\
			\xi_n &\rightharpoonup \xi & &\text{in }L^{\frac{p_1}{p_1-1}}(\partial\Omega),
		\end{aligned}
	\end{equation}
	for some $w\in \mathcal{V}_1$ and $\xi \in L^{\frac{p_1}{p_1-1}}(\partial\Omega)$. Since $(u_1^n,u_2^n)\in \mathcal{W}$ is a solution of \eqref{problem}, we have from \eqref{defsol1} with $\varphi_1=w_n-w\in \mathcal{V}_1$ and the representation $u_1^n=\|u_1^n\|_{1,p_1} w_n$ that
	\begin{equation}\label{defsol3}
		\begin{aligned}
			&\int_\Omega |\nabla w_n|^{p_1-2}\nabla w_n \cdot \nabla (w_n-w)\,\diff x+\int_{\Omega} |w_n|^{p_1-2}w_n(w_n-w)\,\diff x\\
			&=\int_{\partial\Omega} \xi_n(w_n-w)\,\diff \sigma.
		\end{aligned}
	\end{equation}
	From \eqref{defsol3} and \eqref{convergence-properties3} we obtain that
	\begin{align*}
		\lim_{n\to \infty}\int_\Omega |\nabla w_n|^{p_1-2}\nabla w_n \cdot \nabla (w_n-w)\,\diff x=0.
	\end{align*}
	Thus, again by the \textnormal{(S$_+$)}-property of $A_{p_1}$  on $\mathcal{V}_1$ it follows that $w_n \to w$ in $\mathcal{V}_1$ which implies that $w\neq 0$ since $\|w_n\|_{1,p_1}=1$. Moreover, from the strong convergence in $\mathcal{V}_1$ and the fact that
	$(u_1^n,u_2^n)\in \mathcal{W}$ is a solution of \eqref{problem} as well as the representation $u_1^n=\|u_1^n\|_{1,p_1} w_n$ it follows from \eqref{defsol1} that
	\begin{align*}
			&\int_\Omega |\nabla w|^{p_1-2}\nabla w \cdot \nabla \varphi\,\diff x
			+\int_\Omega |w|^{p_1-2}w\varphi\,\diff x
			=\int_{\partial\Omega} \xi\varphi\,\diff \sigma
	\end{align*}
	for all $\varphi\in\mathcal{V}_1$.

	Taking \eqref{H2} and \eqref{H3} into account, for any given $\varepsilon>0$ there exists an integer $n(x)$ for a.a.\,$x\in \partial\Omega$ such that for every $n \geq n(x)$ it holds
	\begin{align*}
		(c_1-\varepsilon)w_n(x)^{p_1-1} &\leq \xi_n(x) \leq (\alpha_1+\varepsilon)w_n(x)^{p_1-1} \quad\text{for a.a.\,} x\in\partial\Omega.
	\end{align*}
	Since $\varepsilon>0$ is arbitrary, letting $n\to \infty$, we get via Mazur's theorem
	\begin{align*}
			c_1w(x)^{p_1-1} &\leq \xi(x)=\mu(x)w(x)^{p_1-1} \leq \alpha_1 w(x)^{p_1-1}\quad\text{for a.a.\,} x\in\partial\Omega
	\end{align*}
	with $c_1\leq \mu(x)\leq \alpha_1$ for a.a.\,$x\in\partial\Omega$ and
	\begin{align*}
		\mu(x)=\frac{g_1(x,u_{1,+}(x),u_{2,+}(x))}{u_{1,+}(x)^{p_1-1}}>0\quad\text{for a.a.\,} x\in\partial\Omega.
	\end{align*}
	Hence $w$ is an eigenfunction associated to the eigenvalue $1$ of the weighted eigenvalue problem with weight $\mu(x)>0$
	\begin{equation}\label{Steklov3}
		\begin{aligned}
			-\Delta_{p_1}w
			& =-w^{p_1-1} \quad &  & \text{in }\Omega, \\
			|\nabla w|^{p_1-2}\nabla w\cdot \nu
			& =\mu(x)w^{p_1-1}     &  & \text{on }\partial\Omega.
		\end{aligned}
	\end{equation}
	We consider now the $V(x)$-weighted eigenvalue problem
	\begin{equation}\label{Steklov4}
		\begin{aligned}
			-\Delta_{p_1} w_V
			& =- w_V^{p_1-1} \quad &  & \text{in }\Omega, \\
			|\nabla  w_V|^{p_1-2}\nabla w_V\cdot \nu
			& =\lambda(V) V(x) w_V^{p_1-1}     &  & \text{on }\partial\Omega,
		\end{aligned}
	\end{equation}
	with $V(x)>0$, $\lambda(V)$ the eigenvalue for the weight $V(x)$ and $w_V$ the corresponding eigenfunctions. In the following, we call $\lambda_1(V)$ the first eigenvalue of \eqref{Steklov4}.  Since $w$ is nonnegative, due to Fern\'{a}ndez Bonder-Rossi \cite[Theorem 1.2 and Proposition 3.1]{Fernandez-Bonder-Rossi-2002}, we know that $\lambda_1(\mu)=1$ because of \eqref{Steklov3}. We consider now problem \eqref{Steklov4} with weights $c_1$ and $\lambda_{1,p_1}$ and related first eigenvalues $\lambda_1(c_1)$ and $\lambda_1(\lambda_{1,p_1})$, respectively. Since $\lambda_{1,p_1}<c_1\leq \mu(x)$ for a.a.\,$x\in\partial\Omega$, we have with \cite[Theorem 1.3]{Fernandez-Bonder-Rossi-2002} that
	\begin{align}\label{Steklov6}
		1=\lambda_1(\mu) \leq \lambda_1(c_1) <\lambda_1(\lambda_{1,p_1}).
	\end{align}
	Since $\lambda_{1,p_1}$ is the smallest eigenvalue of \eqref{Steklov-eigenvalue-problem} with eigenfunction $u_{1,p_1}>0$ we see that $\lambda_1(\lambda_{1,p_1})=1$.
	This is a contradiction to \eqref{Steklov6}. Hence, $u_{i,+}\neq 0$ for $i=1,2$. Since $u_{i,+} \in [0,k_i]$ for $i=1,2$, by the nonlinear regularity theory, see Remark \ref{regularity-argument}, we conclude that $(u_{1,+},u_{2,+})\in \interior\times\interior$.

	It remains to show that $(u_{1,+},u_{2,+})$ is a minimal positive solution of problem \eqref{problem}. To this end, let $(v_1,v_2)\in \mathcal{W}$ be any positive solution of \eqref{problem} such that $v_1 \leq u_{1,+}$ and $v_2\leq u_{2,+}$. Again, by the nonlinear regularity theory and the strong maximum principle, we know that $(v_1,v_2)\in \interior\times\interior$. This fact along with the construction of $(u_{1,+},u_{2,+})$ ensures that
	\begin{align}\label{Steklov8}
		\frac{1}{n}u_{1,p_i} \leq v_i \leq u_{i,+} \leq u_i^n \leq k_i \quad\text{for }i=1,2
	\end{align}
	whenever $n$ is sufficiently large. However, since $(u_1^n,u_2^n)$ is a minimal solution in $[\frac{1}{n}u_{1,p_1},k_1]\times [\frac{1}{n}u_{1,p_2},k_2]$, we get from \eqref{Steklov8} that $u_i^n \leq v_i$ for $i=1,2$. But then, again because of \eqref{Steklov8}, it follows that $u_{1,+}=v_1$ and $u_{2,+}=v_2$. This completes the proof of the theorem.
\end{proof}

\section{Another nontrivial solution}\label{Section4}

In this section we are interested in a third nontrivial solution of the system \eqref{problem} under the assumption that \eqref{problem} has a variational structure. To be more precise we consider the system
\begin{equation}\label{problem2}
	\begin{aligned}
		-\Delta_{p_1}u_1
		& =-|u_1|^{p_1-2}u_1 \quad && \text{in } \Omega,\\
		-\Delta_{p_2}u_2
		& =-|u_2|^{p_2-2}u_2 \quad && \text{in } \Omega,\\
		|\nabla u_1|^{p_1-2}\nabla u_1 \cdot \nu
		&=g_{s_1}(x,u_1,u_2) && \text{on } \partial\Omega,\\
		|\nabla u_2|^{p_2-2}\nabla u_2 \cdot \nu
		&=g_{s_2}(x,u_1,u_2)  && \text{on } \partial\Omega,
	\end{aligned}
\end{equation}
where
\begin{align*}
	(g_1(x,s_1,s_2),g_2(x,s_1,s_2))=(g_{s_1}(x,s_1,s_2),g_{s_2}(x,s_1,s_2))=:\nabla g(x,s_1,s_2),
\end{align*}
with $g\colon \partial\Omega \times \R\times \R \to\R$ being a Carath\'eodory function which is twice differentiable with respect to the second and third variable $(s_1,s_2)\in\R^2$. Moreover, we suppose that the partial derivatives $g_{s_1}, g_{s_2}, g_{s_1s_1}, g_{s_1s_2}, g_{s_2s_2}$ are Carath\'eodory functions on $\partial \Omega \times \R^2$ and $g_{s_1}, g_{s_2}$ are supposed to be bounded on bounded sets. Without any loss of generality, we assume that $g(x,0,0)=0$ for a.a.\,$\partial\Omega$.

To avoid having to write down all the conditions again, let us now assume that \eqref{H0}--\eqref{H3} hold true replacing $(g_1,g_2)$ by $(g_{s_1},g_{s_2})$. Taking Theorem \ref{theorem-minimal-maximal-constant-sign-solutions} into account, we can find a minimal positive solution $(u_{1,+},u_{2,+})$ of problem \eqref{problem2} with $u_{i,+} \leq k_i$ for $i=1,2$. Based on this, we introduce the truncation function $\tau_+\colon \partial\Omega \times \R^2\to\R^2$ assigning to each $(x,s_1,s_2)\in\partial\Omega \times \R^2$ the projection $\tau_+(x,s_1,s_2)$ of $(s_1,s_2)$  on  the closed convex subset $[0,u_{1,+}(x)]\times [0,u_{2,+}(x)]$ of $\R^2$. In the same way, by applying the maximal negative solution $(u_{1,-},u_{2,-})$ of \eqref{problem2} with $u_{i,-}\geq d_i$ for $i=1,2$ obtained in Theorem \ref{theorem-minimal-maximal-constant-sign-solutions}, we define the truncation function $\tau_-\colon\partial\Omega\times\R^2\to\R^2$ as the projection $\tau_-(x,s_1,s_2)$ of $(s_1,s_2)$ on the closed convex subset $[u_{1,-}(x),0]\times [u_{2,-}(x),0]$ of $\R^2$. Lastly, we introduce the truncation function $\tau_0\colon\partial\Omega \times\R^2\to\R^2$ as the projection $\tau_0(x,s_1,s_2)$ of $(s_1,s_2)$ on the closed convex subset $[u_{1,-}(x),u_{1,+}(x)]\times [u_{2,-}(x),u_{2,+}(x)]$ of $\R^2$.

With the help of the truncation functions $\tau_+,\tau_-,\tau_0\colon\partial\Omega \times\R^2\to\R^2$ we can introduce truncated functions related to $g\colon \partial\Omega \times \R\times \R \to\R$ in the following way:
\begin{align*}
	g_+(x,s_1,s_2)
	&=g(x,\tau_+(x,s_1,s_2))\\
	&\quad+(s_1-u_{1,+}(x))^+g_{s_1}(x,\tau_+(x,s_1,s_2))\\
	&\quad+(s_2-u_{2,+}(x))^+g_{s_2}(x,\tau_+(x,s_1,s_2))\\
	&\quad -s_1^-g_{s_1}(x,\tau_+(x,s_1,s_2))\\
	&\quad -s_2^-g_{s_2}(x,\tau_+(x,s_1,s_2)),\\
	g_-(x,s_1,s_2)
	&=g(x,\tau_-(x,s_1,s_2))\\
	&\quad-(s_1-u_{1,-}(x))^-g_{s_1}(x,\tau_-(x,s_1,s_2))\\
	&\quad-(s_2-u_{2,-}(x))^-g_{s_2}(x,\tau_-(x,s_1,s_2))\\
	&\quad +s_1^+g_{s_1}(x,\tau_-(x,s_1,s_2))\\
	&\quad +s_2^+g_{s_2}(x,\tau_-(x,s_1,s_2)),\\
	g_0(x,s_1,s_2)
	&=g(x,\tau_0(x,s_1,s_2))\\
	&\quad-(s_1-u_{1,-}(x))^-g_{s_1}(x,\tau_0(x,s_1,s_2))\\
	&\quad-(s_2-u_{2,-}(x))^-g_{s_2}(x,\tau_0(x,s_1,s_2))\\
	&\quad +(s_1-u_{1,+}(x))^+g_{s_1}(x,\tau_0(x,s_1,s_2))\\
	&\quad +(s_2-u_{2,+}(x))^+g_{s_2}(x,\tau_0(x,s_1,s_2)).
\end{align*}
These truncated mappings $g_-,g_+,g_0\colon \partial\Omega\times\R^2\to\R$ are Carath\'eodory functions being locally Lipschitz continuous with respect to the variables $(s_1,s_2)\in\R^2$. Therefore, their generalized gradients in the sense of Clarke exist. Applying Clarke's calculus according to \cite[Theorem 2.5.1]{Clarke-1990}, we have the following representations:
\begin{equation}\label{representation_gradient_g}
	\begin{aligned}
		&\partial_{(s_1,s_2)} g_+(x,s_1,s_2)=\{\nabla g(x,s_1,s_2)\}\\
		&\text{for a.a.\,}x\in\partial\Omega \text{ and for all }(s_1,s_2)\in [0,u_{1,+}(x)]\times [0,u_{2,+}(x)],
	\end{aligned}
\end{equation}
\begin{equation}\label{representation_gradient_g-2}
	\begin{aligned}
		&\partial_{(s_1,s_2)} g_-(x,s_1,s_2)=\{\nabla g(x,s_1,s_2)\}\\
		&\text{for a.a.\,}x\in\partial\Omega \text{ and for all }(s_1,s_2)\in [u_{1,-}(x),0]\times [u_{2,-}(x),0],
	\end{aligned}
\end{equation}
\begin{equation}\label{representation_gradient_g-3}
	\begin{aligned}
		&\partial_{(s_1,s_2)} g_0(x,s_1,s_2)=\{\nabla g(x,s_1,s_2)\}\\
		&\text{for a.a.\,}x\in\partial\Omega \text{ and for all }(s_1,s_2)\in [u_{1,-}(x),u_{1,+}(x)]\times [u_{2,-}(x),u_{2,+}(x)].
	\end{aligned}
\end{equation}

Taking the modified truncated functions $g_-,g_+,g_0\colon \partial\Omega\times\R^2\to\R$ into account, we introduce the related truncated, nonsmooth functionals $E_+, E_-, E_0\colon \mathcal{W}\to \R$ defined by
\begin{align*}
	E_+(u_1,u_2)&=\frac{1}{p_1} \| u_1\|_{1,p_1}^{p_1} +\frac{1}{p_2} \|u_2\|_{1,p_2}^{p_2}-\int_{\partial \Omega} g_+(x,u_1,u_2) \,\diff\sigma,\\
	E_-(u_1,u_2)&=\frac{1}{p_1} \| u_1\|_{1,p_1}^{p_1} +\frac{1}{p_2} \|u_2\|_{1,p_2}^{p_2}-\int_{\partial \Omega} g_-(x,u_1,u_2) \,\diff\sigma,\\
	E_0(u_1,u_2)&=\frac{1}{p_1} \| u_1\|_{1,p_1}^{p_1} +\frac{1}{p_2} \|u_2\|_{1,p_2}^{p_2}-\int_{\partial \Omega} g_0(x,u_1,u_2) \,\diff\sigma.
\end{align*}
These functionals are locally Lipschitz and so their generalized gradients exist. Before we consider the location of the critical points of these functionals, we need to suppose an additional condition:

\begin{enumerate}[label=\textnormal{(H$_4$)},ref=\textnormal{H$_4$}]
	\item\label{H4}
		\begin{enumerate}
			\item[\textnormal{(i)}]
				The function $s_2\mapsto g_{s_1}(x,s_1,s_2)$ is nondecreasing on the interval $[d_2,k_2]$ for a.a.\,$x\in\partial\Omega$ and for all $s_1\in [d_1,k_1]$.
			\item[\textnormal{(ii)}]
				The function $s_1\mapsto g_{s_2}(x,s_1,s_2)$ is nondecreasing on the interval $[d_1,k_1]$ for a.a.\,$x\in\partial\Omega$ and for all $s_2\in [d_2,k_2]$.
		\end{enumerate}
\end{enumerate}

Next, we are interested in the location of critical points of the functionals $E_+, E_-, E_0\colon \mathcal{W}\to \R$.

\begin{proposition}\label{prop_location_critical_points}
	Let hypotheses \eqref{H0}--\eqref{H3} be satisfied, where $(g_1,g_2)$ is replaced by $\nabla g$ and suppose \eqref{H4}. Then, the following assertions hold:
	\begin{enumerate}
		\item[\textnormal{(i)}]
			If $(v_1,v_2) \in \mathcal{W}$ is a critical point of $E_+$, then
			\begin{align*}
				0\leq v_1(x) \leq u_{1,+}(x)
				\quad\text{and}\quad
				0\leq v_2(x) \leq u_{2,+}(x)
			\end{align*}
			for a.a.\,$x\in\partial\Omega$.
		\item[\textnormal{(ii)}]
			If $(v_1,v_2) \in \mathcal{W}$ is a critical point of $E_-$, then
			\begin{align*}
				u_{1,-}(x) \leq v_1(x)\leq 0
				\quad\text{and}\quad
				u_{2,-}(x)\leq v_2(x) \leq 0
			\end{align*}
			for a.a.\,$x\in\partial\Omega$.
		\item[\textnormal{(iii)}]
			If $(v_1,v_2) \in \mathcal{W}$ is a critical point of $E_0$, then
			\begin{align*}
				u_{1,-}(x) \leq v_1(x)\leq u_{1,+}(x)
				\quad\text{and}\quad
				u_{2,-}(x)\leq v_2(x) \leq u_{2,+}(x)
			\end{align*}
			for a.a.\,$x\in\partial\Omega$.
	\end{enumerate}
\end{proposition}

\begin{proof}
	We only prove the assertion in \textnormal{(i)}, the cases \textnormal{(ii)} and \textnormal{(iii)} can be shown using similar arguments. To this end, let $(v_1,v_2)$ be a critical point of $E_+$, that is, $(0,0)\in \partial E_+(v_1,v_2)$ which means that
	\begin{equation}\label{problem3}
		\begin{aligned}
			-\Delta_{p_1}v_1
			& =-|v_1|^{p_1-2}v_1 \quad && \text{in } \Omega,\\
			-\Delta_{p_2}v_2
			& =-|v_2|^{p_2-2}v_2 \quad && \text{in } \Omega,\\
			|\nabla v_1|^{p_1-2}\nabla v_1 \cdot \nu
			&=h_1(x)&& \text{on } \partial\Omega,\\
			|\nabla v_2|^{p_2-2}\nabla v_2 \cdot \nu
			&=h_2(x) && \text{on } \partial\Omega,
		\end{aligned}
	\end{equation}
	where
	\begin{align}\label{expressions-h}
		(h_1(x),h_2(x)) \in \partial_{(s_1,s_2)}g_+(x,v_1(x),v_2(x)),
	\end{align}
	which follows from Clarke \cite[Theorem 2.7.5]{Clarke-1990}. Note that the precise expressions of the functions $h_1$ and $h_2$ can be found in Carl \cite{Carl-2012-a,Carl-2012-b}. Using the fact that $(u_{1,+}, u_{2,+})$ solves problem \eqref{problem2}, we obtain, due to \eqref{problem3} and \eqref{expressions-h}, by choosing the test function $(v_1-u_{1,+})^+\in \mathcal{V}_1$, that
	\begin{align*}
		\int_\Omega&\left(|\nabla v_1|^{p_1-2}\nabla v_1-|\nabla u_{1,+}|^{p_1-1}\nabla u_{1,+}\right)\cdot\nabla(v_1-u_{1,+})^+dx \\
		&+\int_\Omega  \left(|v_1|^{p_1-2}v_1-u_{1,+}^{p_1-1}\right)(v_1-u_{1,+})^+ \,\mathrm{d}x\\
		&=\int_{\{v_1>u_{1,+}\}} \left(h_1(x)-g_{s_1}(x,u_{1,+},u_{2,+})\right) (v_1-u_{1,+}) \,\mathrm{d}\sigma\\
		&=\int_{\{v_1>u_{1,+}, v_2<0\}} \left(g_{s_1}(x,u_{1,+},0)-g_{s_1}(x,u_{1,+},u_{2,+})\right) (v_1-u_{1,+}) \,\mathrm{d}\sigma\\
		&\quad +\int_{\{v_1>u_{1,+}, 0\leq v_2\leq u_{2,+}\}} \left(g_{s_1}(x,u_{1,+},v_2)-g_{s_1}(x,u_{1,+},u_{2,+})\right) (v_1-u_{1,+}) \,\mathrm{d}\sigma,
	\end{align*}
	since
	\begin{align*}
		\partial_{(s_1,s_2)} g_+(x,s_1,s_2)=\{\nabla g(x,u_{1,+}(x),u_{2,+}(x))\}\quad\text{for a.a.\,}x\in\partial\Omega
	\end{align*}
	provided $s_1>u_{1,+}(x)$ and $s_2>u_{2,+}(x)$. Applying \eqref{H4} (i) gives us
	\begin{align*}
		\int_\Omega&\left(|\nabla v_1|^{p_1-2}\nabla v_1-|\nabla u_{1,+}|^{p_1-1}\nabla u_{1,+}\right)\cdot\nabla(v_1-u_{1,+})^+\,\mathrm{d} x\\
		&+\int_\Omega  \left(|v_1|^{p_1-2}v_1-u_{1,+}^{p_1-1}\right)(v_1-u_{1,+})^+ \,\mathrm{d}x \leq 0.
	\end{align*}
	Therefore, $v_1 \leq u_{1,+}$. In the same way, using \eqref{H4} (ii), we show that $v_2 \leq u_{2,+}$.

	From hypotheses \eqref{H2} and \eqref{H3} we get that $g_{s_1}(x,0,s_2)=0$ for a.a.\,$x\in\partial\Omega$ and for all $s_2 \in [0,k_2]$. Using $-v_1^- \in\mathcal{V}_1$ as test function, taking  \eqref{problem3} and \eqref{expressions-h} again into account, it follows that
	\begin{align*}
		&\int_\Omega|\nabla v_1|^{p_1-2}\nabla v_1\cdot\nabla(-v_1^-)\diff x +\int_\Omega |v_1|^{p_1-2}v_1 \left(-v_1^- \right)\,\mathrm{d} x\\
		&=-\int_{\{v_1 <0\}} h_1v_1\,\mathrm{d}\sigma\\
		&=-\int_{\{v_1<0, v_2<0\}} g_{s_1}(x,0,0)v_1\,\mathrm{d}\sigma
		-\int_{\{v_1<0, 0\leq v_2 \leq u_{2,+}\}} g_{s_1} (x,0,v_2)v_1 \,\mathrm{d}\sigma\\
		&\quad -\int_{\{v_1<0, v_2>u_{2,+}\}} g_{s_1} (x,0,u_{2,+}) v_1 \,\mathrm{d}\sigma=0.
	\end{align*}
	Thus, $v_1 \geq 0$. In the same way, we prove that $v_2\geq 0$.
\end{proof}

In the next proposition we are going to compare the minimal and maximal constant-sign solutions obtained in Theorem \ref{theorem-minimal-maximal-constant-sign-solutions} with the minimizers of the constructed nonsmooth functionals.

\begin{proposition}\label{prop_local_minimizers}
	Let hypotheses \eqref{H0}--\eqref{H3} be satisfied, where $(g_1,g_2)$ is replaced by $\nabla g$ and suppose \eqref{H4}. Then the minimal positive solution $(u_{1,+},u_{2,+})$ of problem \eqref{problem2} is the unique global minimizer of $E_+$ and a local minimizer of $E_0$ while the maximal negative solution $(u_{1,-},u_{2,-})$ of problem \eqref{problem2} is the unique global minimizer of $E_-$ and a local minimizer of $E_0$.
\end{proposition}

\begin{proof}
	Due to the truncated function $g_+\colon \partial\Omega\times\R^2\to\R$, it is clear that the functional $E_+\colon\mathcal{W}\to\R$ is coercive and sequentially weakly lower semicontinuous. This guarantees the existence of a global minimizer $(w_1,w_2) \in\mathcal{W}$ of $E_+$ which is a critical point of $E_+$ in the sense of nonsmooth analysis, see Section \ref{Section2}. From Proposition \ref{prop_location_critical_points} (i) and $u_{i,+} \leq k_i$ for $i=1,2$, it follows that
	\begin{align*}
		0 \leq w_i(x) \leq u_{i,+}(x) \leq k_i\quad\text{for a.a.\,}x\in\partial\Omega \text{ and for }i=1,2.
	\end{align*}
	Recall that $u_{1,p_i}$ is the first eigenfunction of the Steklov eigenvalue problem given in \eqref{Steklov-eigenvalue-problem} with $\|u_{1,p_i}\|_{p_i,\partial\Omega}=1$. This implies that
	\begin{align}\label{steklov-representation}
		\|\nabla u_{1,p_i}\|_{p_i}^{p_i}+\|u_{1,p_i}\|_{p_i}^{p_i}=\lambda_{1,p_i},
	\end{align}
	where $\lambda_{1,p_i}>0$ is the associated first eigenvalue. Then, from hypothesis \eqref{H2} and \eqref{H4} along with the mean value theorem applied to $g_+$, we know that for every $\varepsilon>0$ there exists $t>0$ such that
	\begin{align*}
		E_+(tu_{1,p_1},tu_{1,p_2}) \leq \left(\lambda_{1,p_1}-c_1+\varepsilon \right)\frac{t^{p_1}}{p_1} +\left(\lambda_{1,p_2}-c_2+\varepsilon \right)\frac{t^{p_2}}{p_2},
	\end{align*}
	where we have used \eqref{steklov-representation}. Therefore, taking $\varepsilon>0$ such that $\varepsilon< \min\{c_1-\lambda_{1,p_1},c_1-\lambda_{1,p_1}\}$, we see that $E_+(tu_{1,p_1},tu_{1,p_2})<0$. Therefore, $(w_1,w_2)\neq (0,0)$.

	Next, let us now prove that both components of $w_i$ are nontrivial. Suppose that $w_1\neq 0$ and $w_2=0$. From hypothesis \eqref{H2} we find a number $\delta>0$ small enough such that
	\begin{align*}
		g(x,s_1,s_2)-g(x,s_1,0) > \lambda_{1,p_2} \frac{s_2^{p_2}}{p_2}
	\end{align*}
	for a.a.\,$x\in\partial\Omega$, for all $s_1\in (0,k_1]$ and for all $s_2 \in (0,\delta]$. Using this fact together with $\|u_{1,p_2}\|_{p_2,\partial\Omega}^{p_2}=1$, we obtain for $t>0$ small enough that
	\begin{align*}
		E_+(w_1,tu_{1,p_2})
		&=E_+(w_1,0)+\lambda_{1,p_2}\frac{t^{p_2}}{p_2}\\
		&\quad -\int_{\partial\Omega} \left(g(x,w_1,tu_{1,p_2})-g(x,w_1,0)\right) \,\mathrm{d}\sigma\\
		&<E_+(w_1,0),
	\end{align*}
	which is a contradiction since $(w_1,0)$ is the global minimizer of $E_+$. A similar argument can be used in order to show that $w_2\neq0$. Therefore, we have that $w_1\neq 0$ and $w_2 \neq 0$.

	As a critical point of $E_+$ is understood in the sense of \eqref{problem3} with \eqref{expressions-h}, we know from Proposition \ref{prop_location_critical_points} \textnormal{(i)} along with \eqref{representation_gradient_g} that $(w_1,w_2)$ is a solution of problem \eqref{problem2}. Applying the regularity theory, as before, yields that $(w_1,w_2)\in \interior\times\interior$. Then, combining Proposition \ref{prop_location_critical_points} and the fact that $(u_{1,+},u_{2,+})$ is the minimal positive solution of \eqref{problem2}, we conclude that $(w_1,w_2)=(u_{1,+},u_{2,+})$. Therefore, we know that $(u_{1,+},u_{2,+})$ is a local minimizer of $E_0$ on $C^1(\overline{\Omega})\times C^1(\overline{\Omega})$ since the functionals coincide on $\interior\times\interior$. Then, from Bai-Gasi\'{n}ski-Winkert-Zeng \cite{Bai-Gasinski-Winkert-Zeng-2020}, we know that $(u_{1,+},u_{2,+})$ is a local minimizer of $E_0$ on $\mathcal{W}$. In a similar way, by using (ii) instead of (i) in Proposition \ref{prop_location_critical_points} (2) and \eqref{representation_gradient_g-2} instead of \eqref{representation_gradient_g}, we can show the results of  $(u_{1,-},u_{2,-})$.
\end{proof}

For the next result, we need associated scalar problems of \eqref{problem2} defined by

\begin{equation}\label{scalar-1}
	\begin{aligned}
		-\Delta_{p_1}u_1
		& =-|u_1|^{p_1-2}u_1 \quad && \text{in } \Omega,\\
		|\nabla u_1|^{p_1-2}\nabla u_1 \cdot \nu
		&=g_{s_1}(x,u_1,0) && \text{on } \partial\Omega,\
	\end{aligned}
\end{equation}
and
\begin{equation}\label{scalar-2}
	\begin{aligned}
		-\Delta_{p_2}u_2
		& =-|u_2|^{p_2-2}u_2 \quad && \text{in } \Omega,\\
		|\nabla u_2|^{p_2-2}\nabla u_2 \cdot \nu
		&=g_{s_2}(x,0,u_2)  && \text{on } \partial\Omega.
	\end{aligned}
\end{equation}

We have the following result.

\begin{proposition}\label{prop_auxiliary_problem}
	Let hypotheses \eqref{H0}--\eqref{H3} be satisfied, where $(g_1,g_2)$ is replaced by $\nabla g$ and suppose \eqref{H4}. Then there exists $(u_+,v_+)\in \interior\times\interior$ such that $u_+$ is a solution of \eqref{scalar-1} and $v_+$ is a solution of \eqref{scalar-2} satisfying
	\begin{align*}
		u_+ \leq u_{1,+}, \quad v_+\leq u_{2,+},\quad
		E_+(u_+,0)=\inf E_+(\cdot,0), \quad E_+(0,v_+)=\inf E_+(0,\cdot).
	\end{align*}
	Furthermore, there exists $(u_-,v_-)\in (-\interior)\times(-\interior)$ such that $u_-$ is a solution of \eqref{scalar-1} and $v_-$ is a solution of \eqref{scalar-2} satisfying
	\begin{align*}
		u_- \geq u_{1,-}, \quad v_-\geq u_{2,-},\quad
		E_-(u_-,0)=\inf E_-(\cdot,0), \quad E_-(0,v_-)=\inf E_-(0,\cdot).
	\end{align*}
\end{proposition}

\begin{proof}
	Note that $E_+(\cdot,0)\colon W^{1,p_1}(\Omega)\to \R$ is coercive and sequentially weakly lower semicontinuous. Hence, we can find $u_+\in W^{1,p_1}(\Omega)$ such that
	\begin{align*}
		E_+(u_+,0)=\inf E_+(\cdot,0).
	\end{align*}
	Therefore, $u_+$ is a critical point of $E_+(\cdot,0)$, that is, $0 \in \partial E_+(\cdot,0)(u_+)$. By means of \eqref{H4} we have with nonnegative test function $\varphi_1$
	\begin{align*}
		&\int_\Omega |\nabla u_{1,+}|^{p_1-2}\nabla u_{1,+} \cdot\nabla  \varphi_1\,\mathrm{d}x +\int_{\Omega} u_{1,+}^{p_1-1}\varphi_1\,\mathrm{d}x\\
		&= \int_{\partial\Omega} g_{s_1}(x,u_{1,+},u_{2,+})\varphi_1\,\diff\sigma
		\geq \int_{\partial\Omega} g_{s_1}(x,u_{1,+},0)\varphi_1\,\diff\sigma.
	\end{align*}
Using this and the fact that $u_+$ solves
	\begin{align*}
		&\int_\Omega |\nabla u_{+}|^{p_1-2}\nabla u_{+} \cdot\nabla \varphi_1\,\mathrm{d}x +\int_{\Omega} |u_{+}|^{p_1-2}u_{+}\varphi_1\,\mathrm{d}x
		= \int_{\partial \Omega}
		 (g_+)_{s_1}(x,u_{+},0)\varphi_1\,\mathrm{d} \sigma
	 \end{align*}
	 we get, with the test function $(u_+-u_{1,+})^+\in W^{1,p_1}(\Omega)$, that
	\begin{align*}
		 \int_\Omega&\left(|\nabla u_+|^{p_1-2}\nabla u_+-|\nabla u_{1,+}|^{p_1-2}\nabla u_{1,+}\right)\nabla(u_+-u_{1,+})^+\diff x\\
		 &+\int_\Omega \left (|u_+|^{p_1-2}u_+-u_{1,+}^{p_1-1} \right)(u_+-u_{1,+})^+\,\mathrm{d}x\\
		&\leq \int_{\{u_+>u_{1,+}\}} \left((g_+)_{s_1}(x,u_+,0)-g_{s_1}(x,u_{1,+},u_{2,+})\right) (u_+-u_{1,+}) \,\mathrm{d}\sigma\\
		&=\int_{\{u_+>u_{1,+}\}} \left(g_{s_1}(x,u_{1,+},0)-g_{s_1}(x,u_{1,+},u_{2,+})\right) (u_+-u_{1,+}) \,\mathrm{d}\sigma\leq 0.
	\end{align*}
	This implies $0 \leq u_+(x)\leq u_{1,+}(x)$ for a.a.\,$x\in\partial\Omega$. Therefore, $u_+$ is a solution of \eqref{scalar-1}. Applying again the regularity results, as before, we get that $u_+\in\interior$. The proofs for $v_+$, $u_-$ and $v_-$ can be done in a very similar way.
\end{proof}

For our main result, we need the following sub-homogeneous conditions on the right-hand sides of \eqref{scalar-1} and \eqref{scalar-2}.

\begin{enumerate}[label=\textnormal{(H$_5$)},ref=\textnormal{H$_5$}]
	\item\label{H5}
	For any $t\in [0,1]$ the following hold:
	\begin{enumerate}
		\item[\textnormal{(i)}]
			$g_{s_1}(x,ts_1,0) \leq t^{p_1-1} g_{s_1}(x,s_1,0)$ for a.a.\,$x\in\partial\Omega$ and for all $s_1 \in [d_1,0]$;
		\item[\textnormal{(ii)}]
			$g_{s_1}(x,ts_1,0) \geq t^{p_1-1} g_{s_1}(x,s_1,0)$ for a.a.\,$x\in\partial\Omega$ and for all $s_1 \in [0,k_1]$.
	\end{enumerate}
\end{enumerate}

\begin{enumerate}[label=\textnormal{(H$_6$)},ref=\textnormal{H$_6$}]
	\item\label{H6}
	For any $t\in [0,1]$ the following hold:
	\begin{enumerate}
		\item[\textnormal{(i)}]
		$g_{s_2}(x,0,ts_2) \leq t^{p_2-1} g_{s_2}(x,0,s_2)$ for a.a.\,$x\in\partial\Omega$ and for all $s_2 \in [d_2,0]$;
		\item[\textnormal{(ii)}]
		$g_{s_2}(x,0,ts_2) \geq t^{p_2-1} g_{s_2}(x,0,s_2)$ for a.a.\,$x\in\partial\Omega$ and for all $s_2 \in [0,k_2]$.
	\end{enumerate}
\end{enumerate}

Now we can formulate and prove our main result.

\begin{theorem}
	Let hypotheses \eqref{H0}--\eqref{H3} be satisfied, where $(g_1,g_2)$ is replaced by $\nabla g$ and suppose \eqref{H4}--\eqref{H6}. Moreover, we replace in \eqref{H2} the eigenvalues $\lambda_{1,p_i}$ by $\lambda_{2,p_i}$ for $i=1,2$, where $\lambda_{2,p_i}$ is the second eigenvalue of the $p_i$-Laplacian with Steklov boundary condition. Then, the system \eqref{problem2} has at least three nontrivial solutions, that is, a minimal positive solution
	\begin{align*}
	(u_{1,+},u_{2,+}) \in \interior\times\interior,
	\end{align*}
	a maximal negative solution
	\begin{align*}
		(u_{1,-},u_{2,-})\in (-\interior)\times (-\interior),
	\end{align*}
	and a third solution $(u_{1,0},u_{2,0}) \in C^1(\overline{\Omega})\times C^1(\overline{\Omega})$ such that $(u_{1,0},u_{2,0}) \neq (0,0)$ and
	\begin{align*}
		u_{1,-} \leq u_{1,0} \leq u_{1,+}
		\quad\text{and}\quad
		u_{2,-} \leq u_{2,0} \leq u_{2,+}.
	\end{align*}
\end{theorem}

\begin{proof}
	The existence of a minimal positive solution $(u_{1,+},u_{2,+}) \in \interior\times\interior$ and a maximal negative solution $(u_{1,-},u_{2,-})\in (-\interior)\times (-\interior)$ of \eqref{problem2} follows from Theorem \ref{theorem-minimal-maximal-constant-sign-solutions}. By Proposition \ref{prop_local_minimizers} we know that both pairs $(u_{1,+},u_{2,+}) $ and $(u_{1,-},u_{2,-})$ are local minimizers of the functional $E_0$. Since they are extremal positive and negative solutions of \eqref{problem2}, taking Proposition \ref{prop_location_critical_points} into account, we can suppose that they are strict local minimizers. We also point out that the functional $E_0$ fulfills the nonsmooth Palais-Smale condition (see, for example, Motreanu-R\u{a}dulescu \cite[Definitions 1.5--1.7]{Motreanu-Radulescu-2003}) since $E_0$ is coercive. This allows us to apply the nonsmooth version of the mountain-pass theorem stated in Theorem \ref{mountain-pass-theorem} which gives us a critical point $(u_{1,0},u_{2,0})\in\mathcal{W}$ of $E_0$, that is,
	\begin{align*}
		(0,0) \in \partial E_0 (u_{1,0},u_{2,0})
	\end{align*}
	satisfying
	\begin{equation}\label{sign-changing-1}
		\begin{aligned}
			&\max \left\{ E_0(u_{1,+},u_{2,+}),E_0(u_{1,-},u_{2,-})\right\}\\
			&<E_0(u_{1,0},u_{2,0})=\inf_{\gamma \in \Gamma} \max_{-1\leq t \leq 1} E_0(\gamma(t)),
		\end{aligned}
	\end{equation}
	where
	\begin{align}\label{sign-changing-15}
		\Gamma=\left \{\gamma \in C([0,1],\mathcal{W})\colon \gamma(0)=(u_{1,-},u_{2,-}),\, \gamma(1)=(u_{1,+},u_{2,+}) \right\}.
	\end{align}
	Clearly, from \eqref{representation_gradient_g-3} as well as Proposition \ref{prop_location_critical_points} (iii) and the expression of the generalized gradient $\partial E_0(u_{1,0},u_{2,0})$, we see that $(u_{1,0},u_{2,0})$ is a solution of \eqref{problem2}. Furthermore, because of \eqref{sign-changing-1}, we directly conclude that
	\begin{align*}
		(u_{1,0},u_{2,0})\neq (u_{1,+},u_{2,+})
		\quad\text{and}\quad
		(u_{1,0},u_{2,0})\neq (u_{1,-},u_{2,-}).
	\end{align*}
	It remains to show that $(u_{1,0},u_{2,0}) \neq (0,0)$. The idea is to construct a path $\tilde{\gamma} \in \Gamma$ such that
	\begin{align*}
		E_0(\tilde{\gamma}(t)) <0 \quad\text{for all }t \in [0,1].
	\end{align*}

	From Proposition \ref{prop_auxiliary_problem} we know that $u_+$ and $v_+$ are the positive solutions of \eqref{scalar-1} and \eqref{scalar-2} while $u_-$ and $v_-$ are the negative solutions of \eqref{scalar-1} and \eqref{scalar-2}, respectively. Let us assume that
	\begin{align}\label{sign-changing-4}
		E_+(u_+,0)\leq E_+(0,v_+),
	\end{align}
	the case $E_+(0,v_+)<E_+(u_+,0)$ can be handled in the same way. For $\varepsilon>0$ sufficiently small we set
	\begin{align}\label{sign-changing-5}
		m:=E_+(u_{1,+},u_{2,+})
		\quad\text{and}\quad
		c=E_+(u_+,\varepsilon u_{1,p_2}).
	\end{align}
	Since $(u_{1,+},u_{2,+})$ is the unique global minimizer of $E_+$, see Proposition \ref{prop_local_minimizers}, we see that $m<c$.

	{\bf Claim:} There are no other critical values of $E_+$ in the interval $(m,c]$.

	Due to Proposition \ref{prop_location_critical_points} (i), the representation of the generalized gradient in \eqref{representation_gradient_g} and the fact that $(u_{1,+},u_{2,+})$ is a minimal positive solution of problem \eqref{problem2}, it is clear that we cannot have critical points of $E_+$ whose both components are positive others than $(u_{1,+},u_{2,+})$. Using again hypothesis \eqref{H2} we have for $\varepsilon>0$ small enough that
	\begin{align}\label{sign-changing-2}
		g(x,u_+,\varepsilon u_{1,p_2})-g(x,u_+,0) >\lambda_{1,p_2} \frac{\varepsilon^{p_2}}{p_2} u_{1,p_2}^{p_2}
	\end{align}
	for a.a.\,$x\in\partial\Omega$. From \eqref{sign-changing-2} we obtain, since $\|u_{1,p_2}\|_{p_2,\partial\Omega}^{p_2}=1$, that
	\begin{equation}\label{sign-changing-3}
		\begin{aligned}
			E_+(u_+,\varepsilon u_{1,p_2})
			&=E_+(u_+,0)+\lambda_{1,p_2}\frac{\varepsilon^{p_2}}{p_2}\\
			&\quad -\int_{\partial\Omega} \left(g(x,u_+,\varepsilon u_{1,p_2})-g(x,u_+,0)\right) \,\mathrm{d}\sigma\\
			&<E_+(u_+,0).
	\end{aligned}
	\end{equation}
	Therefore, from \eqref{sign-changing-3}, \eqref{sign-changing-4}, \eqref{sign-changing-5} and Proposition \ref{prop_auxiliary_problem} we conclude that there are no critical values of $E_+$ in the interval $(m,c]$ associated to critical points with one positive component and the other one equal to zero. This proves the Claim.

	Because of the Claim, we can now apply the nonsmooth version of the second deformation lemma to the functional $E_+$, see Gasi\'{n}ski-Papageorgiou \cite[Theorem 2.1.1]{Gasinski-Papageorgiou-2005}. This gives us a continuous map $\eta=(\eta_1,\eta_2)\colon [0,1]\times E^{-1}_+((-\infty,c])\to E^{-1}_+((-\infty,c])$ such that
	\begin{equation}\label{sign-changing-6}
		\begin{aligned}
			&\eta(0,u_1,u_2)=(u_1,u_2), \quad \eta(1,u_1,u_2)=(u_{1,+},u_{2,+}),\\
			& E_+(\eta(t,u_1,u_2)) \leq E_+(u_1,u_2)
		\end{aligned}
	\end{equation}
	for all $t \in [0,1]$ and for all $(u_1,u_2) \in E^{-1}_+((-\infty,c])$. Based on \eqref{sign-changing-6}, we define a path $\gamma_+\in C([0,1],\mathcal{W})$ by
	\begin{align*}
		\gamma_+(t)=(\eta_1(t,u_+,\varepsilon u_{1,p_2})^+,\eta_2(t,u_+,\varepsilon u_{1,p_2})^+)\quad\text{for all }t \in[0,1].
	\end{align*}
	Obviously, the path $\gamma_+$ joins $(u_+,\varepsilon u_{1,p_2})$ and $(u_{1,+},u_{2,+})$. Taking \eqref{sign-changing-6} and \eqref{sign-changing-3} into account, we derive that
	\begin{equation}\label{sign-changing-8}
		\begin{aligned}
			E_0(\gamma_+(t))
			&=E_+(\gamma_+(t))
			\leq E_+(\eta_1(t,u_+,\varepsilon u_{1,p_2})^+,\eta_2(t,u_+,\varepsilon u_{1,p_2})^+)\\
			&\leq E_+(u_+,\varepsilon u_{1,p_2})
			< E_+(u_+,0) \leq E_+(0,v_+)
		\end{aligned}
	\end{equation}
	for all $t\in [0,1]$ and for $\varepsilon>0$ sufficiently small.

	Next, we can suppose, without any loss of generality, that
	\begin{align*}
		E_-(u_-,0) \leq E_-(0,v_-).
	\end{align*}
	Then, as above, we can construct a path $\gamma_- \in C([0,1],\mathcal{W})$ such that $\gamma_-(0)=(u_-,-\varepsilon u_{1,p_2})$, $\gamma_-(1)=(u_{1,-},u_{2,-})$ and
	\begin{equation}\label{sign-changing-9}
		\begin{aligned}
			E_0(\gamma_-(t))< E_-(u_-,0) \leq E_-(0,v_-)<0
		\end{aligned}
	\end{equation}
	for all $t\in [0,1]$ and for $\varepsilon>0$ sufficiently small.

	Now, let $S_i=W^{1,p_i}(\Omega) \cap \partial B_1^{p_i, \partial \Omega}$ with $\partial B^{p_i, \partial \Omega}_1=\{u \in L^{p_i}(\partial\Omega)\colon \|u\|_{p_i, \partial \Omega}=1\}$ be endowed with the topology induced by $W^{1,p_i}(\Omega)$ for $i=1,2$ and let $S_{i,C}=S_i\cap C^1(\overline{\Omega})$ be endowed with the topology induced by $C^1(\overline{\Omega})$. We set
	\begin{align*}
		\Gamma_{0,i}&=\left\{\gamma \in C([0,1],S_i)\colon \gamma(0)=-u_{1,p_i}, \, \gamma(1)=u_{1,p_i} \right\},\\
		\Gamma_{0,i,C}&=\left\{\gamma \in C([0,1],S_{i,C})\colon \gamma(0)=-u_{1,p_i},\, \gamma(1)=u_{1,p_i}\right\}
	\end{align*}
	for $i=1,2$.

	Now let us fix constants $\tilde{\mu} \in (0,c_2-\lambda_{2,p_2})$ and $\hat{\mu}\in (0,c_2-\lambda_{2,p_2}-\mu)$ with $c_2$ as in \eqref{H2}. Then, the density of $S_{2,C}$ in $S_2$ (which implies the density of $\Gamma_{0,2,C}$ in $\Gamma_{0,2}$, see Winkert \cite{Winkert-PHD-2009} for a proof of it), guarantees that we can find a path $\gamma_{0,2}\in\Gamma_{0,2,C}$ such that
	\begin{align}\label{sign-changing-10}
		\max_{u\in \gamma_{0,2}([0,1])} \|u\|_{1,p_2}^{p_2} <\lambda_{2,p_2}+\hat{\mu}.
	\end{align}
	Note that we supposed hypothesis \eqref{H2} with $\lambda_{1,p_i}$ replaced by $\lambda_{2,p_i}$. Then it follows that we can find $\delta>0$ such that
	\begin{align}\label{sign-changing-11}
		g(x,s_1,s_2) -g(x,s_1,0) > (c_2-\tilde{\mu}) \frac{s_2^{p_2}}{p_2}
	\end{align}
	for a.a.\,$x\in\partial\Omega$, for all $s_1\in [d_1,k_1]$ and for all $s_2\in  (0,\delta)$. Now, we can choose $\varepsilon>0$ sufficiently small such that
	\begin{align}\label{sign-changing-12}
		\varepsilon|\gamma_{0,2}(t)(x)|<\delta\quad\text{for all }t\in [0,1]\text{ and for a.a.\,}x\in\partial\Omega.
	\end{align}
	Combining \eqref{sign-changing-10}, \eqref{sign-changing-11}, \eqref{sign-changing-12} and the fact that $\|\gamma_{0,2}(t)\|_{p_2,\partial\Omega}^{p_2}=1$ for all $t\in [0,1]$, we obtain
	\begin{equation}\label{sign-changing-13}
		\begin{aligned}
			E_0(v,\varepsilon \gamma_{0,2}(t))
			&=\frac{1}{p_1} \|v\|_{1,p_1}^{p_1}+ \frac{\varepsilon^{p_2}}{p_2} \|\gamma_{0,2}(t)\|_{1,p_2}^{p_2}-\int_{\partial\Omega} g(x,v,\varepsilon \gamma_{0,2}(t))\,\diff\sigma\\
			&=E(v,0)+ \frac{\varepsilon^{p_2}}{p_2} \|\gamma_{0,2}(t)\|_{1,p_2}^{p_2}\\
			&\quad+\int_{\partial\Omega} \left(g(x,v,0)-g(x,v,\varepsilon \gamma_{0,2}(t))\right) \,\diff\sigma\\
			& \leq E(v,0)+\frac{\varepsilon^{p_2}}{p_2} \left(\lambda_{2,p_2}+\hat{\mu}-c_2+\tilde{\mu}\right)
		\end{aligned}
	\end{equation}
	for all $t\in [0,1]$ and for all $v\in W^{1,p_1}(\Omega)$ with $v\in [u_{1,-},u_{2,+}]$. Now we take a continuous path $\gamma_1\colon [0,1]\to C^1(\overline{\Omega})$ such that $\gamma_1(0)=u_-$, $\gamma_1(1)=u_+$ and we set $\gamma_0(t)=(\gamma_1(t),\varepsilon \gamma_{0,2}(t))$. Then we get a path with the endpoints $(u_-,-\varepsilon u_{1,p_2})$ and $(u_+,\varepsilon u_{1,p_2})$ such that, due to \eqref{sign-changing-13},
	\begin{align}\label{sign-changing-14}
		E_0(\gamma_0(t)) \leq E_0(\gamma_1(t),0)+\frac{\varepsilon^{p_2}}{p_2} \left(\lambda_{2,p_2}+\hat{\mu}-c_2+\tilde{\mu}\right)
	\end{align}
	for all $t\in [0,1]$. The concatenation of the paths $\gamma_-, \gamma_0$ and $\gamma_+$ generates a path $\tilde{\gamma}$ which satisfies, because of \eqref{sign-changing-8}, \eqref{sign-changing-9}, and \eqref{sign-changing-14}
	\begin{align*}
		E_0(\tilde{\gamma}(t)) \leq \max_{t\in[0,1]}E_0(\gamma_1(t),0)+\frac{\varepsilon^{p_2}}{p_2} \left(\lambda_{2,p_2}+\hat{\mu}-c_2+\tilde{\mu}\right)
	\end{align*}
	for all $t\in [0,1]$. From \eqref{sign-changing-1} and \eqref{sign-changing-15} we see that
	\begin{align}\label{sign-changing-16}
		E_0(u_{1,0},u_{2,0}) \leq \max_{t\in[0,1]} E_0(\gamma_1(t),0)+\frac{\varepsilon^{p_2}}{p_2} \left(\lambda_{2,p_2}+\hat{\mu}-c_2+\tilde{\mu}\right).
	\end{align}
	Recall that $\hat{\mu}\in (0,c_2-\lambda_{2,p_2}-\mu)$. Therefore,
	\begin{align}\label{sign-changing-17}
		\frac{\varepsilon^{p_2}}{p_2} \left(\lambda_{2,p_2}+\hat{\mu}-c_2+\tilde{\mu}\right)<0.
	\end{align}
	This means, with regard to \eqref{sign-changing-16} and \eqref{sign-changing-17}, we only have to prove the existence of a continuous path $s\mapsto \gamma_1(s)$ with $\gamma_1(0)=u_-$ and $\gamma_1(1)=u_+$ satisfying
	\begin{align}\label{sign-changing-24}
		E_0(\gamma_1(s),0)\leq 0 \quad\text{for all }s\in [0,1].
	\end{align}

	We define the path $\gamma_1$ by
	\begin{align*}
		\gamma_1(s)=
		\begin{cases}
			(1-2s)u_- &\text{if }s \in \left[0,\frac{1}{2}\right],\\[1ex]
			(2s-1)u_+ & \text{if }s\in \left[\frac{1}{2},1\right].
		\end{cases}
	\end{align*}
	Applying $g_0(\cdot,(1-2s)u_-,0)=g(\cdot,(1-2s)u_-,0)$, we get for $s\in [0,\frac{1}{2}]$
	\begin{align}\label{sign-changing-18}
		E_0(\gamma_1(s),0)=\frac{1}{p_1}(1-2s)^{p_1} \|u_-\|_{1,p_1}^{p_1}-\int_{\partial\Omega} g(x,(1-2s)u_-,0)\,\diff\sigma.
	\end{align}
	Since $u_-$ is a solution of \eqref{scalar-1}, it holds
	\begin{align}\label{sign-changing-19}
		\|u_-\|_{1,p_1}^{p_1}=\int_{\partial\Omega} g_{s_1}(x,u_-,0)u_-\,\diff\sigma.
	\end{align}
	Combining \eqref{sign-changing-18} and \eqref{sign-changing-19} yields
	\begin{equation}\label{sign-changing-20}
		\begin{aligned}
			&E_0(\gamma_1(s),0)\\
			&=\int_{\partial\Omega} \left(\frac{1}{p_1}(1-2s)^{p_1} g_{s_1}(x,u_-,0)u_--g(x,(1-2s)u_-,0)\right)\,\diff\sigma.
		\end{aligned}
	\end{equation}
	We observe that
	\begin{equation}\label{sign-changing-21}
		\begin{aligned}
			&\int_{\partial\Omega} \frac{1}{p_1}(1-2s)^{p_1} g_{s_1}(x,u_-,0)u_-\,\diff\sigma\\
			&=\int_{\partial\Omega} \int_0^1\frac{\partial}{\partial t}\frac{t^{p_1}}{p_1}(1-2s)^{p_1}g_{s_1}(x,u_-,0)u_-\,\mathrm{d}t\diff\sigma\\
			&=\int_{\partial\Omega} \int_0^1t^{p_1-1}(1-2s)^{p_1}g_{s_1}(x,u_-,0)u_-\,\mathrm{d}t\diff\sigma
		\end{aligned}
	\end{equation}
	and
	\begin{equation}\label{sign-changing-22}
		\begin{aligned}
			\int_{\partial\Omega} g(x,(1-2s)u_-,0)\,\diff\sigma
			&=\int_{\partial\Omega}\int_0^1\frac{\partial}{\partial t} g(x,t(1-2s)u_-,0)\,\mathrm{d}t\diff\sigma\\
			&=\int_{\partial\Omega}\int_0^1 g_{s_1}(x,t(1-2s)u_-,0)(1-2s)u_-\,\mathrm{d}t\diff\sigma.
		\end{aligned}
	\end{equation}
	Using \eqref{sign-changing-21} and \eqref{sign-changing-22} in \eqref{sign-changing-20} and hypothesis \eqref{H5} (i) leads to
	\begin{equation*}
		\begin{aligned}
			&E_0(\gamma_1(s),0)\\
			&=\int_{\partial\Omega}\int_0^1 (1-2s)u_-\left((t(1-2s))^{p_1-1}g_{s_1}(x,u_-,0)-g_{s_1}(x,t(1-2s)u_-,0)\right)\,\mathrm{d}t\diff\sigma\\
			&\leq 0.
		\end{aligned}
	\end{equation*}
	Using similar arguments, one can show that $E_0(\gamma_1(s),0)\leq 0$ for $[\frac{1}{2},1]$. Hence, we have shown that \eqref{sign-changing-24} is satisfied. This completes the proof.
\end{proof}

\begin{example}
	For the sake of simplicity, we have omitted the $x$-dependence on $g$  and consider the problem
	\begin{equation*}
		\begin{aligned}
			-\Delta_{p_1}u_1
			& =-|u_1|^{p_1-2}u_1 \quad && \text{in } \Omega,\\
			-\Delta_{p_2}u_2
			& =-|u_2|^{p_2-2}u_2 \quad && \text{in } \Omega,\\
			|\nabla u_1|^{p_1-2}\nabla u_1 \cdot \nu
			&=-\alpha(p_1+q_1)|u_1|^{p_1+q_1-2}u_1 & & \\
			&\quad+\beta p_1(u_1^+)^{p_1-1}(u_2^+)^{p_2}+\gamma p_1|u_1|^{p_1-2}u_1&& \text{on } \partial\Omega,\\
			|\nabla u_2|^{p_2-2}\nabla u_2 \cdot \nu
			&=-\alpha(p_2+q_2)|u_2|^{p_2+q_2-2}u_2 &&\\
			&\quad+\beta p_2(u_2^+)^{p_2-1}(u_1^+)^{p_1}+\gamma p_2|u_2|^{p_2-2}u_2  && \text{on } \partial\Omega,
		\end{aligned}
	\end{equation*}
	with constants $p_1, p_2\geq 2$, $\alpha,\beta, q_1,q_2>0$ and
	\begin{align*}
		\gamma>\max\left\{\frac{\lambda_{1,p_1}}{p_1},\frac{\lambda_{2,p_2}}{p_2}\right\}.
	\end{align*}
	Then, the potential is given by
	\begin{align*}
		g(s_1,s_2)=-\alpha\left(|s_1|^{p_1+q_1}+|s_2|^{p_2+q_2}\right)+\beta(s_1^+)^{p_1}(s_2^+)^{p_2}+\gamma\left(|s_1|^{p_1}+|s_2|^{p_2}\right)
	\end{align*}
	with the partial derivatives
	\begin{align*}
		g_{s_1}(s_1,s_2)&=g_1(s_1,s_2)\\
		&=-\alpha(p_1+q_1)|s_1|^{p_1+q_1-2}s_1+\beta p_1(s_1^+)^{p_1-1}(s_2^+)^{p_2}+\gamma p_1|s_1|^{p_1-2}s_1,\\
		g_{s_2}(s_1,s_2)&=g_2(s_1,s_2)\\
		&=-\alpha(p_2+q_2)|s_2|^{p_2+q_2-2}s_2+\beta p_2(s_1^+)^{p_1}(s_2^+)^{p_2-1}+\gamma p_2|s_2|^{p_2-2}s_2.
	\end{align*}
	Then, for any constants $k_1, k_2>0$ and $d_1, d_2<0$, Hypotheses \eqref{H0}--\eqref{H6} are satisfies provided $\alpha>0$ is sufficiently large. Let us prove this for $g_1$, the same arguments can be used for $g_2$.

	Since $g_1$ is continuous in $(s_1,s_2)\in\R\times\R$, it is a Carath\'eodory function. Furthermore we have with
	\begin{align*}
		s:=\max_{(s_1,s_2)\in M}\{|s_1|,|s_2|\},
	\end{align*}
	where $M$ is a bounded set, that
	\begin{align*}
		|g_1(s_1,s_2)|&\le \alpha(p_1+q_1)s^{p_1+q_1-1}+\beta p_1 s^{p_1+p_2-1}+\gamma p_1 s^{p_1-1}\\
		&:=C\in L^\infty(\partial\Omega).
	\end{align*}

	Furthermore, for $t_1,s_1,t_2,s_2\in[-K_1,K_1]$, we have
	\begin{equation}\label{EstimateExample}
		\begin{aligned}
		&|g_1(s_1,t_1)-g_1(s_2,t_2)|\\
		&\le\alpha(p_1+q_1)||s_1|^{p_1+q_1-2}s_1-|s_2|^{p_1+q_1-2}s_2|\\
		&\quad+\beta p_1|(s_1^+)^{p_1-1}(t_1^+)^{p_2}-(s_2^+)^{p_1-1}(t_2^+)^{p_2}|\\
		&\quad+\gamma p_1||s_1|^{p_1-2}s_1-|s_2|^{p_1-2}s_2|\\
		&\le	\alpha(p_1+q_1)||s_1|^{p_1+q_1-2}s_1-|s_2|^{p_1+q_1-2}s_2|\\
		&\quad+\beta p_1\Big(|(s_1^+)^{p_1-1}(t_1^+)^{p_2} -(s_1^+)^{p_2-1}(t_2^+)^{p_2}|\\
		&\qquad \qquad
		+|(s_1^+)^{p_2-1}(t_2^+)^{p_2}-(s_2^+)^{p_1-1}(t_2^+)^{p_2}|\Big)\\
		&\quad+\gamma p_1||s_1|^{p_1-2}s_1-|s_2|^{p_1-2}s_2|\\
		&\le\alpha(p_1+q_1)||s_1|^{p_1+q_1-2}s_1-|s_2|^{p_1+q_1-2}s_2|\\
		&\quad+\beta p_1\Big(|(s_1^+)|^{p_1-1}|(t_1^+)^{p_2} -(t_2^+)^{p_2}|
		+|(t_2^+)^{p_2}||(s_1^+)^{p_2-1}-(s_2^+)^{p_1-1}|\Big)\\
		&\quad+\gamma p_1||s_1|^{p_1-2}s_1-|s_2|^{p_1-2}s_2|\\
		&\le\alpha(p_1+q_1)||s_1|^{p_1+q_1-2}s_1-|s_2|^{p_1+q_1-2}s_2|\\
		&\quad+\beta p_1\Big(K_1^{p_1-1}|(t_1^+)^{p_2} -(t_2^+)^{p_2}|
		+K_1^{p_2}|(s_1^+)^{p_2-1}-(s_2^+)^{p_1-1}|\Big)\\
		&\quad+\gamma p_1||s_1|^{p_1-2}s_1-|s_2|^{p_1-2}s_2|.
		\end{aligned}
	\end{equation}
	Now we consider the function $f(x):=|x|^Px$ for $x\in[-K_1,K_1]$. Because of $f'(x)=(P+1)|x|^P$ we see that $f$ is continuously differentiable with $\sup|f'(x)|=(P+1)K_1^P$, so it is a Lipschitz constant. Hence
	\begin{align*}
		|f(x_1)-f(x_2)|\le (P+1)K_1^P|x_1-x_2|.
	\end{align*}
	This can be used for the first and third term in \eqref{EstimateExample}. For the second term we consider $\tilde f(x)=|x|^Q$ for $x\in[-K_1,K_1]\setminus\{0\}$. In a similar way it then can be shown that $sup \tilde f'(x)=\sup Q|x|^{Q-2}x\le QK_1^{Q-1}$ and therefore
	\begin{align*}
		|\tilde f(x_1)-\tilde f(x_2)|\le QK_1^{Q-1}|x_1-x_2|.
	\end{align*}
	This helps to estimate the second term in \eqref{EstimateExample} for $s_1,s_2,t_1,t_2\in[-K_1,K1]\setminus\{0\}$. In the case where $t_1=t_2=0$ or $t_1>0$ and $t_2<0$ we directly get $|(t_1^+)^{p_2}+(t_2^+)^{p_2}|\le |t_1-t_2|$. So, all together yields
	\begin{align*}
		&|g_1(s_1,t_1)-g_1(s_2,t_2)|\\
		&\le\alpha(p_1+q_1)||s_1|^{p_1+q_1-2} s_1-|s_2|^{p_1+q_1-2}s_2|\\
		&\quad+\beta p_1\Big(K_1^{p_1-1}| (t_1^+)^{p_2} -(t_2^+)^{p_2}|
		+K_1^{p_2}|(s_1^+)^{p_2-1} -(s_2^+)^{p_1-1}|\Big)\\
		&\quad+\gamma p_1||s_1|^{p_1-2} s_1-|s_2|^{p_1-2}s_2|\\
		&\le \alpha(p_1+q_1)(p_1+q_1-2) K_1^{p_1+q_1-2}|s_1-s_2|\\
		&\quad+\beta p_1\Big(K_1^{p_1-1}p_2K_1^{p_2-1}|t_1-t_2| +K_1^{p_2}(p_2-1)K_1^{p_2-2}|s_1-s_2|\Big)\\
		&\quad+\gamma p_1(p_1-2+1)K_1^{p_1-2}|s_1-s_2|\\
		&\le L_1(|s_1-s_2|+|t_1-t_2|)
	\end{align*}
	for $L_1>0$ sufficiently large. This shows \eqref{H0}.

	Let $k_1, k_2>0$ and $s_2\in[0,k_2]$. Then
	\begin{align*}
		g_1(k_1,s_2)&=-\alpha(p_1+q_1)k_1^{p_1+q_1-1}+\beta p_1 k_1^{p_1-1}(s_2^+)^{p_2}+\gamma p_1 k_1^{p_1-1}\\
		&\le -\alpha(p_1+q_1)k_1^{p_1+q_1-1}+\beta p_1 k_1^{p_1-1}k_2^{p_2}+\gamma p_1 k_1^{p_1-1}\\
		&\le 0
	\end{align*}
	for $\alpha>0$ sufficiently large. Let $d_1, d_2<0$ and $s_2\in[d_2,0]$. Then
	\begin{align*}
		g_1(x,d_1,s_2)=-\alpha(p_1+q_1)|d_1|^{p_1+q_1-2}d_1+\gamma p_1|d_1|^{p_1-2}d_1\ge0
	\end{align*}
	for $\alpha>0$ sufficiently large. This shows \eqref{H1}.

	For every $s_2\in(0,k_2]$ we have that
	\begin{align*}
		&\liminf_{s_1\to0^+}\frac{g_1(s_1,s_2)}{s_1^{p_1-1}}\\
		&=\liminf_{s_1\to0^+}\frac{-\alpha(p_1+q_1)|s_1|^{p_1+q_1-2}s_1+\beta p_1(s_1^+)^{p_1-1}(s_2^+)^{p_2}+\gamma p_1|s_1|^{p_1-2}s_1}{s_1^{p_1-1}}\\
		&=\liminf_{s_1\to0^+}\left(-\alpha(p_1+q_1)s_1^{q_1}+\beta p_1(s_2^+)^{p_2}+\gamma p_1\right)\\
		&=\beta p_1(s_2^+)^{p_2}+\gamma p_1\\
		&=\limsup_{s_1\to0^+}\frac{g_1(s_1,s_2)}{s_1^{p_1-1}}.
	\end{align*}
	Therefore there exist constants $\hat c_1, \hat \alpha_1>0$ with $\hat\alpha_1\ge\hat c_1>\lambda_{1,p_1}$ such that
	\begin{align*}
		\hat c_1\le \beta p_1(s_2^+)^{p_2}+\gamma p_1\le\hat\alpha_1.
	\end{align*}
	Furthermore, for every $s_2\in [d_2,0)$, we have
	\begin{align*}
		&\liminf_{s_1\to0^-}\frac{g_1(s_1,s_2)}{|s_1|^{p_1-2}s_1}\\
		&=\liminf_{s_1\to0^-}\frac{-\alpha(p_1+q_1)|s_1|^{p_1+q_1-2}s_1+\gamma p_1|s_1|^{p_1-2}s_1}{|s_1|^{p_1-2}s_1}\\
		&=\liminf_{s_1\to0^-}\left(-\alpha(p_1+q_1)|s_1|^{q_1}+\gamma p_1\right)\\
		&=\gamma p_1\\
		&=\limsup_{s_1\to0^-}\frac{g_1(s_1,s_2)}{|s_1|^{p_1-2}s_1}.
	\end{align*}
	Thus, there exist constants $\tilde c_1, \tilde \alpha_1>0$ with $\tilde\alpha_1\ge\tilde c_1>\lambda_{1,p_1}$ such that
	\begin{align*}
		\tilde c_1\le \gamma p_1\le\tilde\alpha_1.
	\end{align*}
	We now set
	\begin{align*}
		c_1:=\min\{\hat c_1,\tilde c_2\}\quad\text{and}\quad \alpha_1:=\max\{\hat\alpha_1,\tilde\alpha_2\}
	\end{align*}
	to get \eqref{H2} and \eqref{H3}.

	Let $x\in \partial\Omega$ and $s_1\in[d_1,k_1]$. We want to show that $g_{s_1}(s_1,\cdot)$ is nondecreasing on the interval $[d_2,k_2]$. For this let $\tilde s_2, \hat s_2\in[d_2,k_2]$ with $\tilde s_2\le \hat s_2$. Then
	\begin{align*}
		&g_{s_1}(s_1,\tilde s_2)\\
		&=-\alpha(p_1+q_1)|s_1|^{p_1+q_1-2}s_1+\beta p_1(s_1^+)^{p_1-1}(\tilde s_2^+)^{p_2}+\gamma p_1|s_1|^{p_1-2}s_1\\
		&\le -\alpha(p_1+q_1)|s_1|^{p_1+q_1-2}s_1+\beta p_1(s_1^+)^{p_1-1}(\hat s_2^+)^{p_2}+\gamma p_1|s_1|^{p_1-2}s_1\\
		&= g_{s_1}(s_1,\hat s_2).
	\end{align*}
	This shows \eqref{H4}.

	Let $t\in[0,1]$. Then we have
	\begin{align*}
		&g_{s_1}(ts_1,0)\\
		&=-\alpha(p_1+q_1)t^{p_1+q_1-1}|s_1|^{p_1+q_1-2}s_1+\gamma p_1t^{p_1-1}|s_1|^{p_1-2}s_1\\
		&= t^{p_1-1}\left(-\alpha(p_1+q_1)t^{q_1}|s_1|^{p_1+q_1-2}s_1+\gamma p_1|s_1|^{p_1-2}s_1\right).
	\end{align*}
	Therefore
	\begin{align*}
		g_{s_1}(ts_1,0) &\le t^{p_1-1}g_{s_1}(\cdot, s_1,0)\quad\text{for } d_1\le s_1\le0,\\
		g_{s_1}(ts_1,0) &\ge t^{p_1-1}g_{s_1}(\cdot, s_1,0),\quad\text{for } 0\le s_1\le k_1.
	\end{align*}
	This shows \eqref{H5}. Hypothesis \eqref{H6} can be shown in a similar way.
\end{example}


\end{document}